\documentclass{amsart}
\usepackage[latin1]{inputenc}
\usepackage{amsmath}
\usepackage{amsthm}
\usepackage{amsfonts}
\usepackage{amssymb}
\usepackage[T1]{fontenc}
\usepackage{float}
\usepackage[all]{xy}

\newtheorem*{Theoquantized}{Theorem 1}
\newtheorem*{Theopolyreg}{Theorem 2}
\newtheorem*{TheopolyA}{Theorem 3}

\newtheorem{theorem}{Theorem}[section]
\newtheorem{prop}[theorem]{Proposition}
\newtheorem{lem}[theorem]{Lemma}
\newtheorem{corol}[theorem]{Corollary}
\newtheorem{conj}[theorem]{Conjecture}

\theoremstyle{definition}
\newtheorem{defi}[theorem]{Definition}

\newtheorem{rmq}[theorem]{Remark}
\newtheorem{exmp}[theorem]{Example}

\def\u{\underline}
\def\Trop{{\rm{Trop}\,}}
\def\op{{\rm{op}\,}}

\def\coker{{\rm{coker}\,}}
\def\Gr{{\rm{Gr}}}

\def\dim{{\rm{dim}\,}}
\def\ddim{{\textbf{dim}\,}}
\def\rep{{\rm{rep}}}

\def\<{\left<}
\def\>{\right>}

\def\ens#1{\left\{ #1 \right\}}
\def\fl{{\longrightarrow}\,}

\def\A{{\mathbb{A}}}
\def\C{{\mathbb{C}}}
\def\N{{\mathbb{N}}}
\def\P{{\mathbb{P}}}
\def\Q{{\mathbb{Q}}}
\def\Z{{\mathbb{Z}}}

\def\CC{{\mathcal{C}}}

\def\ens#1{\left\{ #1 \right\}}
\def\Ext{{\rm{Ext}}}

\def\Hom{{\rm{Hom}}}
\def\Ob{{\rm{Ob}}}
\def\Alin#1{\overrightarrow{\mathbb A}_{#1}}
\def\ker{{\rm{ker}}\,}
\def\coker{{\rm{coker}}\,}
\def\XQy{X^{Q,\textbf y}}
\def\XAy{X^{A,\textbf y}}
\def\Aaffine{\tilde{\mathbb A}}

\title[Quantized Chebyshev polynomials]{Quantized Chebyshev polynomials and cluster characters with coefficients}
\author{G. Dupont}

\begin{document}

\begin{abstract}
	We introduce quantized Chebyshev polynomials as deformations of generalized Chebyshev polynomials previously introduced by the author in the context of acyclic coefficient-free cluster algebras. We prove that these quantized polynomials arise in cluster algebras with principal coefficients associated to acyclic quivers of infinite representation types and equioriented Dynkin quivers of type $\mathbb A$. We also study their interactions with bases and especially canonically positive bases in affine cluster algebras.
\end{abstract}

\maketitle

\setcounter{tocdepth}{1}
\tableofcontents

\begin{section}{Introduction}
	Normalized Chebyshev polynomials are elementary well known objects which can be defined as follows. For every $n \geq 1$, the $n$-th normalized Chebyshev polynomial of the first kind $F_n$ is characterized by $F_n(t+t^{-1})=t^n+t^{-n}$ and the $n$-th normalized Chebyshev polynomial of the second kind $S_n$ is characterized by $S_n(t+t^{-1})=\sum_{k=0}^n t^{n-2k}$. These polynomials made their first apparition in the context of cluster algebras respectively in \cite{shermanz} and in \cite{CZ}.

	Cluster algebras were introduced in early 2000's by Fomin and Zelevinsky in \cite{cluster1}. Since, they found applications in many areas of mathematics like combinatorics, Lie theory, Poisson geometry or representation theory. In their most simple incarnation, cluster algebras are commutative algebras over $\Z\P$ where $\P$ is some tropical semi-field. The generators, called \emph{cluster variables}, are gathered into sets of fixed finite cardinality called \emph{clusters}. Monomials in variables belonging to a same cluster are called \emph{cluster monomials}. The elements of a cluster algebra can always be expressed as Laurent polynomials in cluster variables belonging  to any fixed cluster, this is referred to as the \emph{Laurent phenomenon} \cite{cluster1}. An element in a cluster algebra $\mathcal A$ is called \emph{positive} if it can be expressed as a Laurent polynomial with coefficients in $\Z_{\geq 0}\P$ in every cluster of $\mathcal A$. A $\Z\P$-basis $\mathcal B$ in $\mathcal A$ is called \emph{canonically positive}, if positive elements in $\mathcal A$ coincide precisely with $\Z_{\geq 0}\P$-linear combinations of elements of $\mathcal B$. It is not known if there is necessarily a canonically positive basis in a cluster algebra $\mathcal A$. Nevertheless, if such a basis exists, its elements are uniquely determined up to normalization by elements of $\P$. Canonically positive bases were investigated in particular cases in \cite{shermanz, Cerulli:A21}.

	The research of bases, and especially canonically positive bases, in cluster algebras  was one of the main motivation for their study. In the symmetric \emph{coefficient-free} case, that is, when $\P=\ens 1$, Caldero and Keller proved that if the cluster algebra $\mathcal A$ is simply-laced of finite type (ie if it has only finitely many cluster variables), then cluster monomials form a $\Z$-basis in $\mathcal A$ \cite{CK1}. 

	For rank 2 cluster algebras of affine and finite type, Sherman and Zelevinsky managed to compute canonically positive bases with arbitrary coefficients \cite{shermanz}. In particular, the authors proved that if $\mathcal A$ is a coefficient-free rank 2 cluster algebra of affine type, the canonically positive basis of $\mathcal A$ is $\mathcal B(\mathcal A)=\ens{\textrm{cluster monomials}}\sqcup \ens{F_n(z)|n \geq 1}$ where $z$ is some well chosen particular positive element in $\mathcal A$ (see Section \ref{section:bases} for details). Using coefficient-free cluster characters, Caldero and Zelevinsky managed to compute a slightly different basis for the coefficient-free cluster algebra associated to the Kronecker quiver \cite{CZ}. Namely, this basis is given by $\ens{\textrm{cluster monomials}}\sqcup \ens{S_n(z)|n \geq 1}$. The presence of normalized second kind Chebyshev polynomials in this case comes from the study of coefficient-free cluster characters associated to regular modules over the path algebra of the Kronecker quiver.

	Let $Q=(Q_0,Q_1)$ be an acyclic quiver, that is, a quiver without oriented cycles where $Q_0$ is a finite set of vertices and $Q_1$ a finite set of arrows. Let $k=\C$ be the field of complex numbers, we denote by $kQ$ the path algebra of $Q$, by $kQ$-mod the category of finite dimensional left-$kQ$-modules and by $\mathcal C_Q$ the cluster category of $Q$. 

	Let $\textbf y$ be a $Q_0$-tuple of elements of $\P$. We denote by $\mathcal A(Q,\textbf y,\textbf x)$ the cluster algebra with principal coefficients at the initial seed $(Q,\textbf y, \textbf x)$ where $\textbf x=(x_i| i \in Q_0)$ and $\textbf y=(y_i| i \in Q_0)$ is a minimal set of generators of the semifield $\P$. 

	Inspired by works on cluster characters for the coefficient-free case \cite{CC,CK1,CK2,Palu}, Fu and Keller introduced in \cite{FK} cluster characters with coefficients in order to realize elements in the cluster algebra $\mathcal A(Q,\textbf y,\textbf x)$ from objects in the cluster category $\mathcal C_Q$. In particular cluster variables in $\mathcal A(Q,\textbf y, \textbf x)$ are characters associated to indecomposable rigid objects in the cluster category $\mathcal C_Q$. In this paper, we consider a more elementary description of cluster characters with coefficients than the one proposed in \cite{FK}. We will see in Section \ref{ssection:addingcoeffs} that these two definitions coincide. 
	The \emph{cluster character with coefficients on $\CC_Q$} is a map
	$$\XQy_?:\Ob(\CC_Q) \fl \Z[\textbf y][\textbf x^{\pm 1}]$$
	whose detailed definition will be given in section \ref{section:characters}. We denote by 
	$$X^Q_?:\Ob(\CC_Q) \fl \Z[\textbf x^{\pm 1}]$$
	the usual Caldero-Chapoton map introduced in \cite{CC,CK1}, which will also be referred to as the \emph{cluster character without coefficients on $\CC_Q$}. 

	In \cite{Dupont:stabletubes} (see also \cite{Propp:frieze} for a similar description in Dynkin type $\mathbb A$), we introduced a generalization of Chebyshev polynomials of the second kind arising in cluster algebras associated to acyclic representation-infinite quivers. More precisely, if $Q$ is a quiver of infinite representation type, the coefficient-free cluster character $X^Q_M$ of an indecomposable regular module $M$ can be expressed as a polynomial with integral coefficients evaluated at the characters of quasi-composition factors of $M$. The polynomials appearing were called \emph{generalized Chebyshev polynomials}.

	In \cite{Cerulli:A21}, the author gets interested in cluster algebras with coefficients associated to an affine quiver of type $\Aaffine_{2,1}$. It turned out that if the coefficients are not specialized at 1, generalized Chebyshev polynomials do not appear anymore. The aim of this paper is to introduce a certain deformation of generalized Chebyshev polynomials that allows to recover the polynomiality property for cluster characters with coefficients evaluated at indecomposable regular modules over the path algebra of a representation-infinite quiver.

	Whereas the final goal of this paper is to give an efficient tool for calculations in cluster algebras, most of the results can be read independently of the theory of cluster algebras.

	Our main results are the following~: Consider a family $\textbf q=\ens{q_i|i \in \Z}$ of indeterminates over $\Z$ and a family $\ens{x_{i,1}|i \in \Z}$ of indeterminates over $\Z[\textbf q]$. We define by induction a family $$\ens{x_{i,n} |i \in \Z, n \geq 1} \subset \Q(\textbf q)(x_{i,1} |i \in \Z)$$
	by
	\begin{equation}\label{eq:recidentity}
	      x_{i,n}x_{i+1,n}=x_{i,n+1}x_{i+1,n-1}+\prod_{k=1}^n q_{i+k}
	\end{equation}
	with the convention that $x_{i,0}=1$ for all $i \in \Z$.

	The first result of this paper is a polynomial closed expression for the $x_{i,n}$~:
	\begin{Theoquantized}
	      For any $n \geq 1$ and any $i \in \Z$, we have
	      $$x_{i,n}=\det \left[\begin{array}{ccccccc}
		      x_{i+n-1,1} & 1 &&& (0)\\
		      q_{i+n-1} & \ddots & \ddots \\
		      & \ddots & \ddots & \ddots \\
		      & & \ddots & \ddots & 1 \\
		      (0)& & & q_{i+1} & x_{i,1}
	      \end{array}\right].$$
	      In particular, $x_{i,n}$ is a polynomial in $\Z[q_{i+1}, \ldots, q_{i+n-1}, x_{i,1}, \ldots, x_{i+n-1,1}]$.
	\end{Theoquantized}
	Note that the well-known Dodgson's determinant evaluation rule turns out to be a consequence of theorem 1 and equation (\ref{eq:recidentity}) when all the $q_i'$s are specialized at 1.

	Identifying naturally the ring $\Z[q_{i+1}, \ldots, q_{i+n-1}, x_{i,1}, \ldots, x_{i+n-1,1}]$ with a subring of  $q_{i}, \ldots, q_{i+n-1}, x_{i,1}, \ldots, x_{i+n-1,1}$, we denote by $P_{n}$ the polynomial in $2n$ variables such that
	$$x_{i,n}=P_{n}(q_i, \ldots, q_{i+n-1}, x_{i,1}, \ldots, x_{i+n-1,1})$$
	and $P_{n}$ is called the \emph{$n$-th quantized Chebyshev polynomial of infinite rank}. Note that the definition of $P_{n}$ does not depend on $i$.

	For any $p \geq 1$, the abelian group $p\Z$ acts $\Z$-linearly on $\Z[q_i, x_{i,1} |i \in \Z]$ by $kp.q_i=q_{i+kp}$ and $kp.x_i=x_{i+kp}$ for any $k \in \Z$. We denote by 
	$$\pi_p:\Z[q_i, x_{i,1} |i \in \Z] \fl \Z[q_i, x_{i,1} |i \in \Z]/p\Z$$
	the canonical map. We set $P_{n,p}$ to be the unique polynomial such that for every $i \in \Z$ and $n \geq 1$, we have 
	$$\pi_p(x_{i,n})=P_{n,p}(\pi_p(q_i), \ldots, \pi_p(q_{i+p-1}),\pi_p(x_i), \ldots, \pi_p(x_{i+p-1})).$$
	$P_{n,p}$ is called the \emph{$n$-th quantized Chebyshev polynomial of rank $p$}.
	If we denote by $k[p]$ the remainder of the euclidean division of an integer $k$ by $p$, $P_{n,p}$ is the polynomial such that 
	$P_{n,p}(q_{i[p]}, \ldots, q_{i+p-1[p]},x_{i[p],1}, \ldots, x_{i+p-1[p],1})$ is the determinant $$\det \left[\begin{array}{ccccccc}
	      x_{i+n-1[p],1} & 1 &&& (0)\\
	      q_{i+n-1[p]} & \ddots & \ddots \\
	      & \ddots & \ddots & \ddots \\
	      & & \ddots & \ddots & 1 \\
	      (0)& & & q_{i+1[p]} & x_{i[p],1}
	\end{array}\right]$$

	In the sequel, we will use the following notation~: if $J$ is a set, $\textbf a=\ens{a_i | i \in J}$ is a family of indeterminates over $\Z$ and $\nu=\ens{\nu_i | i \in J} \subset \Z$ has finite support, we write $\textbf a^\nu=\prod_{i \in J} a_i^{\nu_i}$.

	If $Q$ is a representation-infinite quiver, any regular component $\mathcal R$ in the Auslander-Reiten quiver $\Gamma(kQ)$ of $kQ$-mod is of the form $\Z\A_\infty/(p)$ for some $p \geq 0$ \cite[Section VIII.4, Theorem 4.15]{ARS}. We denote by $R_i, i \in \Z/p\Z$ the quasi-simple modules in $\mathcal R$, ordered such that $\tau R_i \simeq R_{i-1}$ for all $i \in \Z/p\Z$. For $i \in \Z/p\Z$ and $n \geq 1$, denote by $R_i^{(n)}$ the unique indecomposable module such that there exists a sequence of irreducible monomorphisms
	$$R_i \simeq R_i^{(1)} \fl R_i^{(2)} \fl \cdots \fl R_i^{(n)}.$$ 
	We say that $R_i^{(n)}$ has quasi-socle $R_i$ and quasi-length $n$. By convention $R_i^{(0)}$ denotes the zero module. The quotients $R_i^{(k)}/R_{i}^{(k-1)}$ for $k=1, \ldots, n$ are called the \emph{quasi-composition factors} of the module $M$. Every indecomposable module in $\mathcal R$ can be written $R_i^{(n)}$ for some $i \in \Z/p\Z$ and $n \geq 1$. 

	Our main result is that quantized Chebyshev polynomials appear naturally for cluster characters with coefficients associated to regular modules. 
	\begin{Theopolyreg}
	      Let $Q$ be a quiver of infinite representation type, $\mathcal R$ be a regular component in $\Gamma(kQ)$ and let $p \geq 0$ be such that $\mathcal R$ is of the form $\Z\A_\infty/(p)$. We denote by $\ens{R_i| i \in \Z/p\Z}$ the set of quasi-simple modules in $\mathcal R$, ordered such that $\tau R_i \simeq R_{i-1}$ for all $i \in \Z/p\Z$. Then for every $n \geq 1$ and $i \in \Z/p\Z$, we have
	      $$\XQy_{R_i^{(n)}}=P_{n}(\textbf y^{\ddim R_{i}}, \ldots, \textbf y^{\ddim R_{i+n-1}}, \XQy_{R_i}, \ldots, \XQy_{R_{i+n-1}})$$
	      or equivalently
	      $$\XQy_{R_i^{(n)}}=\det \left[\begin{array}{ccccccc}
		      \XQy_{R_{i+n-1}} & 1 &&& (0)\\
		      \textbf y^{\ddim R_{i+n-1}} & \ddots & \ddots \\
		      & \ddots & \ddots & \ddots \\
		      & & \ddots & \ddots & 1 \\
		      (0)& & & \textbf y^{\ddim R_{i+1}} & \XQy_{R_i}
	      \end{array}\right].$$
	      Moreover, if $p>0$, we have 
	      $$\XQy_{R_i^{(n)}}=P_{n,p}(\textbf y^{\ddim R_{i}}, \ldots, \textbf y^{\ddim R_{i+p-1}}, \XQy_{R_i}, \ldots, \XQy_{R_{i+p-1}}).$$
	\end{Theopolyreg}

	We also prove that quantized Chebyshev polynomials arise in cluster algebras of Dynkin type $\mathbb A$. For any integer $r \geq 1$, let $A$ be the quiver of type $\Alin{r}$, that is, of Dynkin type $\mathbb A_r$ equipped with the following linear orientation~:
	$$\xymatrix{
		0 & \ar[l] 1 & \ar[l] 2 & \ar[l] \cdots & \ar[l] r-1
	}$$
	Let $\mathcal A(A,\textbf x,\textbf y)$ be the cluster algebra with principal coefficients at the initial seed $(A,\textbf x,\textbf y)$ and $X^{A,\textbf y}$ be the cluster character with coefficients on $\CC_{A}$. 
	
	For any $i \in [0,r-1]$, we denote by $S_i$ the simple $kA$-module associated to the vertex $i$ and for any $n \in [1, r-i]$, we denote by $S_i^{(n)}$ the indecomposable $kA$-module with socle $S_i$ and length $n$. 
	We prove~:
	\begin{TheopolyA}
		Let $r \geq 1$ be an integer and $A$ be the above quiver of type $\Alin{r}$. Then, for any $i \in [0,r-1]$ and $n \in [1,r-i]$, we have
		$$
			X^{A,\textbf y}_{S_i^{(n)}}=P_n(y_i, \ldots, y_{i+n-1}, X^{A,\textbf y}_{S_i}, \ldots, X^{A,\textbf y}_{S_{i+n-1}})
		$$
		or equivalently
		$$X^{A,\textbf y}_{S_i^{(n)}}=\det \left[\begin{array}{ccccccc}
		      X^{A,\textbf y}_{S_{i+n-1}} & 1 &&& (0)\\
		      y_{i+n-1} & \ddots & \ddots \\
		      & \ddots & \ddots & \ddots \\
		      & & \ddots & \ddots & 1 \\
		      (0)& & & y_{i+1} & X^{A,\textbf y}_{S_i}
	      \end{array}\right].$$
	\end{TheopolyA}
	Note that this result was obtained independently by Yang and Zelevinsky by considering generalized minors \cite{YZ:ppalminors}.

	The paper is organized as follows. In section \ref{section:characters}, we give definitions and properties of cluster characters with and without coefficients. In section \ref{section:typeA}, we study in detail cluster characters with coefficients for equioriented Dynkin quivers of type $\A$. The study of Dynkin type $\A$ allows to define quantized Chebyshev polynomials in section \ref{section:quantized} where Theorem 1 and Theorem 3 are proved. In section \ref{section:polyreg}, we prove Theorem 2 and give some explicit computations in cluster algebras of type $\Aaffine_{2,1}$. In section \ref{section:firstandsecond}, we study algebraic properties of some particular quantized Chebyshev polynomials, namely the quantized versions of normalized Chebyshev polynomials of the first and second kinds. Finally, in section \ref{section:bases}, we give examples and conjectures for these polynomials to appear in bases, and especially canonically positive bases, in cluster algebras of affine types.
\end{section}

\begin{section}{Cluster characters}\label{section:characters}
	\begin{subsection}{Definitions and basic properties}
	      Let $Q$ be an acyclic quiver. We denote by $\mathcal A(Q,\textbf x, \textbf y)$ the cluster algebra with principal coefficients at the initial seed $(Q, \textbf x, \textbf y)$ where $\textbf y=\ens{y_i | i \in Q_0}$ is the initial coefficient $Q_0$-tuple and where $x=\ens{x_i | i \in Q_0}$ is the initial cluster. We simply denote by $\mathcal A(Q, \textbf x)$ the coefficient-free cluster algebra with initial seed $(Q,\textbf x)$.

	      Let $k=\C$ be the field of complex numbers and $kQ$-mod be the category of finite dimensional left-modules over the path algebra of $Q$. All along this paper, this category will be identified with the category $\rep(Q)$ of finite dimensional representations of $Q$ over $k$. We denote by $\tau_{kQ}$ (or simply $\tau$) the Auslander-Reiten translation on $kQ$-mod. Let $\mathcal D=D^b(kQ)$ be the bounded derived category of $Q$ with shift functor denoted by $[1]_{kQ}$ (or simply $[1]$). We denote by $\mathcal C_Q$ the cluster category of the quiver $Q$, that is, the orbit category $\mathcal D/F$ of the auto-functor $F=\tau^{-1}[1]$ in $\mathcal D$. This is an additive triangulated category \cite{K}, 2-Calabi-Yau whose indecomposable objects are given by indecomposable $kQ$-modules and shifts of indecomposable projective $kQ$-modules \cite{BMRRT}. This category was independently introduced by Caldero, Chapoton and Schiffler for the type $\A$ case \cite{CCS1}.
	      
	      For every $i \in Q_0$, we denote by $S_i$ the simple $kQ$-module associated to the vertex $i$, $P_i$ its projective cover and $I_i$ its injective hull. We denote by $\alpha_i=\ddim S_i$ the dimension vector of $S_i$. Since $\ddim$ induces an isomorphism of abelian groups $K_0(kQ) \fl \Z^{Q_0}$, $\alpha_i$ is identified with the $i$-th vector of the canonical basis of $\Z^{Q_0}$.
	      
	      As $Q$ is acyclic, $kQ$ is a finite dimensional hereditary algebra, we denote by $\<-, -\>$ the Euler form on $kQ$-mod. It is given by
	      $$\<M,N\>=\dim \Hom_{kQ}(M,N)-\dim \Ext^1_{kQ}(M,N)$$
	      for any $kQ$-modules $M$ and $N$. Note that $\<-,-\>$ is well-defined on the Grothendieck group. 
	      
	      For any $kQ$-module $M$ and any dimension vector $\textbf e$, we denote by 
	      $$\Gr_{\textbf e}(M) =\ens{N \subset M | \ddim N=\textbf e}$$
	      the \emph{grassmannian of submodules of dimension $\textbf e$ of $M$}. This is a projective variety and we denote by $\chi(\Gr_{\textbf e}(M))$ its Euler characteristic with respect to the simplicial cohomology.
	      
	      Roughly speaking, a \emph{cluster character} evaluated at a $kQ$-module $M$ is some normalized generating series for Euler characteristics of grassmannians of submodules of the module $M$. More precisely~:
	      \begin{defi}\label{defi:characters}
		      The \emph{cluster character with coefficients} on $kQ$-mod is the map $\Ob(\mathcal C_Q) \fl \Z[\textbf y][\textbf x^{\pm 1}]$ defined as follows~:
		      \begin{enumerate}
			      \item[a.] If $M$ is an indecomposable $kQ$-module, we set
				      \begin{equation}\label{eq:XMy}
					      X^{Q,\textbf y}_M=\sum_{\textbf e \in \N^{Q_0}} \chi(\Gr_{\textbf e}(M)) \prod_{i \in Q_0} x_i^{-\<\textbf e,\alpha_i\>-\<\alpha_i, \ddim M-\textbf e\>}y_i^{e_i};
				      \end{equation}
			      \item[b.] if $M \simeq P_i[1]$ is the shift of an indecomposable projective module, we set
				      $$X^{Q,\textbf y}_M=x_i;$$
			      \item[c.] for any two objects $M,N$ in $\mathcal C_Q$, we set 
				      $$X^{Q,\textbf y}_MX^{Q,\textbf y}_N=X^{Q,\textbf y}_{M \oplus N}.$$
		      \end{enumerate}
	      \end{defi}
	      It follows from the definition that equation (\ref{eq:XMy}) holds for any $kQ$-module. Note that cluster characters are invariant on isoclasses.
	      
	      For any object $M$ in $\mathcal C_Q$, we denote by $X^Q_M$ the value of the Caldero-Chapoton map at $M$. Equivalently, $X^Q_M$ is the specialization of $X^{Q,\textbf y}_M$ at $y_i=1$ for all $i \in Q_0$.
	      
	      We now prove a multiplication formula on almost split sequences for $\XQy_?$. This is an analogue to \cite[Proposition 3.10]{CC} for the Caldero-Chapoton map.
	      \begin{prop}\label{prop:multps}
		      Let $Q$ be an acyclic quiver, $N$ be an indecomposable non-projective module. Then
		      $$\XQy_M\XQy_N=\XQy_B+\textbf y^{\ddim N}$$
		      where $B$ is the unique $kQ$-module such that there exists an almost split sequence
		      $$0 \fl M \xrightarrow{i} B \xrightarrow{p} N \fl 0.$$
	      \end{prop}
	      \begin{proof}
		      The proof is almost the same as in \cite{CC} for the coefficient-free case. We give it for completeness. We write $\textbf m=\ddim M$ and $\textbf n=\ddim N$. We thus have
		      $$\XQy_M\XQy_N=\XQy_{M \oplus N}
			      =\sum_{\textbf e \in \N^{Q_0}}\chi(\Gr_{\textbf e}(M \oplus N)) \prod_i x_i^{-\<\textbf e, \alpha_i\>-\<\alpha_i, \textbf m+\textbf n-\textbf e\>} y_i^{e_i}.$$
		      Since the varieties $\Gr_{\textbf e}(M \oplus N)$ and $\bigsqcup_{\textbf f + \textbf g=\textbf e}\Gr_{\textbf f}(M) \times \Gr_{\textbf g}(N)$ are isomorphic, 
		      we get
		      $$\XQy_{M\oplus N}
			      =\sum_{\textbf f,\textbf g}\chi(\Gr_{\textbf f}(M))\chi(\Gr_{\textbf g}(N)) \prod_i x_i^{-\<\textbf f+\textbf g, \alpha_i\>-\<\alpha_i, \textbf m+\textbf n-\textbf f-\textbf g\>} y_i^{f_i+g_i}.$$
		      We now consider the case where $\textbf f=0$ and $\textbf g=\ddim N$. Since 
		      $$\Gr_{0}(M) \times \Gr_{\ddim N}(N)=\ens{(0,N)}$$
		      the corresponding Laurent monomial in $\XQy_{M\oplus N}$ is 
		      $$\prod_i x_i^{-\<\textbf n, \alpha_i\>-\<\alpha_i, \textbf m\>} y_i^{n_i}$$
		      but $\textbf m=c(\textbf n)$ where $c$ is the Coxeter transformation induced on $K_0(kQ)$ by the Auslander-Reiten translation. Thus, $\<\textbf n, \alpha_i\>=-\<\alpha_i, \textbf m\>$ and then 
		      $$\prod_i x_i^{-\<\textbf n, \alpha_i\>-\<\alpha_i, \textbf m\>} y_i^{n_i}=\prod_i y_i^{n_i}=\textbf y^{\ddim N}.$$
		      
		      Now, since the sequence is almost split, for every $\textbf e \in \N^{Q_0}$, the map 
		      $$\zeta_{\textbf e}:\left\{\begin{array}{rcl}
			      \Gr_{\textbf e}(B) & \fl & \bigsqcup_{\textbf f + \textbf g=\textbf e} \Gr_{\textbf f}(M) \times \Gr_{\textbf g}(N)\\
			      L & \mapsto & (i^{-1}(L), p(L))
		      \end{array}\right.$$
		      is an algebraic homomorphism such that the fiber of a point $(A,C)$ is empty if and only if $(A,C)=(0,N)$ and is an affine space otherwise. It thus follows that 
		      $$\XQy_M\XQy_N=\XQy_B+\textbf y^{\ddim N}$$
		      and the proposition is proved.
	      \end{proof}
	\end{subsection}

	\begin{subsection}{Adding coefficients to cluster characters}\label{ssection:addingcoeffs}
	      We now prove that in order to compute cluster characters with coefficients associated to a quiver $Q$, it suffices to compute cluster characters without coefficients for a certain $\widehat{Q}$ obtained from the quiver $Q$.
	      More precisely let $Q=(Q_0,Q_1)$ be an acyclic quiver, we denote by $\widehat{Q}=({\widehat Q}_0,{\widehat Q}_1)$ the acyclic quiver with vertex set consisting of two copies of $Q_0$. The first copy is identified with $Q_0$ and the second copy is denoted by 
	      $$Q_0'=\ens{\sigma(v) |v \in Q_0}$$
	      where $\sigma$ is a fixed bijection $Q_0 \fl Q_0'$.
	      For any $v \neq w \in {\widehat{Q}}_0$, if $v,w \in Q_0$, the arrows from $v$ to $w$ in $\widehat Q_1$ are given by the arrows from $v$ to $w$ in $Q_1$, otherwise there are only arrows $v \fl \sigma(v)$ in $\widehat Q_1$ where $v$ runs over $Q_0$. In particular, we can identify $Q_1$ with a subset of $\widehat Q_1$. The quiver $\widehat{Q}$ is called the \emph{framed quiver associated $Q$}. By construction, the framed quiver of an acyclic quiver is itself acyclic. Note that framed quiver are familiar objects in the context of quiver varieties (see e.g.\cite{Nakajima:varietieswquivers}). 
	      
	      Given an acyclic quiver $R=(R_0,R_1)$ we denote by $B(R)$ the incidence matrix of $R$. That is the skew-symmetric matrix $(b_{ij}) \in M_{R_0}(\Z)$ whose entries are given by
	      $$b_{ij}=|\ens{\alpha:i \fl j \in R_1}|-|\ens{\alpha:j \fl i \in R_1}|$$
	      for any $i,j \in R_0$.
	      
	      We thus have
	      $$B(\widehat Q)=\left[\begin{array}{cc}
		      B(Q) & I\\
		      -I & 0
	      \end{array}\right].$$
	      
	      The category $kQ$-mod can be canonically identified with a subcategory of $k\widehat{Q}$-mod. We denote by $\iota: kQ\textrm{-mod} \fl k\widehat{Q}\textrm{-mod}$ the corresponding embedding, realizing $kQ$-mod as a full, exact, extension-closed subcategory of $k\widehat{Q}$-mod. Dimension vectors induce bijections $K_0(kQ\textrm{-mod}) \simeq \Z^{Q_0}$ and $K_0(k\widehat{Q}\textrm{-mod}) \simeq \Z^{\widehat Q_0} $. Identifying $\Z^{Q_0}$ with $\Z^{Q_0} \times \ens{0} \subset \Z^{Q_0} \times \Z^{Q'_0} \simeq \Z^{\widehat Q_0}$ we can identify $K_0(kQ\textrm{-mod})$ with a subgroup of $K_0(k\widehat{Q}\textrm{-mod})$.
	      
	      Let $\mathcal A(\widehat Q, \textbf u)$ be the coefficient-free cluster algebra with initial seed $(\widehat Q, \textbf u)$ where $\textbf u=\ens{u_i | i \in \widehat Q_0}$. According to the Laurent phenomenon, it is a subring of the ring $\Z[\textbf u^{\pm 1}]$ of Laurent polynomials in $\textbf u$. We denote by $X^{\widehat Q}_?: \Ob(\CC_{\widehat Q}) \fl \Z[\textbf u^{\pm 1}]$ the Caldero-Chapoton map on $\CC_{\widehat Q}$. For any $k{\widehat{Q}}$-module $M$, the value of the Caldero-Chapoton map at $M$ is thus given by~:
	      $$X^{\widehat{Q}}_M=\sum_{\textbf e \in \N^{\widehat Q_0}} \chi(\Gr_{\textbf e}(M)) \prod_{i \in \widehat Q_0} u_i^{-\<\textbf e,\alpha_i\>-\<\alpha_i, \ddim M-\textbf e\>}$$ 
	      where $\<-,-\>$ denotes the Euler form on $k\widehat{Q}$-mod. 
	      
	      We consider the homomorphism of $\Z$-algebras
	      $$\pi:\left\{ 
	      \begin{array}{rcll}
		      \Z[\textbf u^{\pm 1}] & \fl & \Z[\textbf y^{\pm 1},\textbf x^{\pm 1}]\\
		      u_i & \mapsto & x_i &\textrm{ if } i \in Q_0,\\
		      u_j & \mapsto & y_i &\textrm{ if } j=\sigma(i) \in Q_0'.\\
	      \end{array}\right.$$
	      
	      \begin{lem}\label{lem:addingcoeffs}
		      For any $kQ$-module $M$, we have
		      $$\XQy_M=\pi\left(X^{\widehat{Q}}_{\iota(M)}\right).$$
	      \end{lem}
	      \begin{proof}
		      Let $M$ be a $kQ$-module which we consider as a representation of $Q$. For any $i \in Q_0$,  we denote by $M(i)$ the corresponding $k$-vector space at vertex $i$ and for any $\alpha:i \fl j \in Q_1$, we denote by $M(\alpha):M(i) \fl M(j)$ the corresponding $k$-linear map. Thus, $\iota(M)$ can be identified with the representation of $\widehat{Q}$ given by $\iota(M)(i)=M(i)$ if $i \in Q_0$, $\iota(M)(i) =0$ if $i \in Q_0'$ and $\iota(M)(\alpha)=M(\alpha)$ if $\alpha \in Q_1$, $\iota(M)(\alpha)=0$ if $\alpha \not \in Q_1$. In particular, $\ddim \iota(M)=\ddim M$. Moreover, $\Gr_{\textbf e}(\iota(M))=\emptyset$ if $\textbf e \not \in \N^{Q_0}$ and $\iota$ induces an isomorphism $\Gr_{\textbf e}(M) \simeq \Gr_{\textbf e}(\iota(M))$ otherwise.
		      
		      Note also that for any $i \in Q_0'$, $j \in Q_0$, we have $\<\alpha_i, \alpha_j\>=0$,
		      $$\<\alpha_j, \alpha_i\>=\left\{\begin{array}{ll}
			      -1 & \textrm{ if }j=\sigma^{-1}(i),\\
			      0 & \textrm{ otherwise,}
		      \end{array}\right.$$
		      and for any $i,j \in Q_0$, the form $\<\alpha_i,\alpha_j\>$ is the same computed in $kQ$-mod and $k\widehat{Q}$-mod.
	      
		      We thus have~:
		      \begin{align*}
			      X^{\widehat{Q}}_M
				      &=\sum_{\textbf e \in \widehat{Q}_0} \chi(\Gr_{\textbf e}(\iota(M))) \prod_{i \in \widehat{Q}_0} u_i^{-\<\textbf e, \alpha_i\>-\<\alpha_i, \ddim \iota(M)-\textbf e\>}\\
				      &=\sum_{\textbf e \in \N^{Q_0}} \chi(\Gr_{\textbf e}(M)) \prod_{i \in Q_0} u_i^{-\<\textbf e, \alpha_i\>-\<\alpha_i, \ddim M-\textbf e\>}\prod_{i \in Q_0'} u_i^{-\<\textbf e, \alpha_i\>-\<\alpha_i, \ddim \iota(M)-\textbf e\>}\\
				      &=\sum_{\textbf e \in \N^{Q_0}} \chi(\Gr_{\textbf e}(M)) \prod_{i \in Q_0} u_i^{-\<\textbf e, \alpha_i\>-\<\alpha_i, \ddim M-\textbf e\>}\prod_{i \in Q_0'} u_i^{-\<\textbf e, \alpha_i\>}\\
				      &=\sum_{\textbf e \in \N^{Q_0}} \chi(\Gr_{\textbf e}(M)) \prod_{i \in Q_0} u_i^{-\<\textbf e, \alpha_i\>-\<\alpha_i, \ddim M-\textbf e\>}\prod_{i \in Q_0'} u_i^{e_{\sigma^{-1}(i)}}\\
		      \end{align*}
		      Applying $\pi$, we thus get
		      \begin{align*}
			      \pi(X^{\widehat{Q}}_M)
				      &=\sum_{\textbf e \in \N^{Q_0}} \chi(\Gr_{\textbf e}(M)) \prod_{i \in Q_0}x_i^{-\<\textbf e, \alpha_i\>-\<\alpha_i, \ddim M-\textbf e\>}\prod_{i \in Q_0'}y_i^{e_i}\\
				      &=\XQy_{M}
		      \end{align*}
		      and the lemma is proved.
	      \end{proof}
	      
	      \begin{rmq}
		      In \cite{FK}, the authors gave a slightly different definition of the cluster characters with coefficients than the one we use here. We now prove that the definition we give in this paper is compatible with their definition. 
		      
		      Let $Q$ be an acyclic quiver, $\widehat{Q}$ the corresponding framed quiver and $\widetilde{Q}=\widehat{Q}^{\op}$. Let mod-$k\widetilde{Q}$ be the category of finite dimensional right modules over $k\widetilde{Q}$ considered in \cite{FK}. This category is equivalent to the category $k\widehat{Q}$-mod of finite dimensional left-modules over the path algebra of $\widehat{Q}$. It thus follows from \cite{FK} that the cluster category $\CC_Q$ is equivalent to the category $^\bot(\Sigma (k\widehat Q/kQ))/(k\widehat Q/kQ)$ where $\Sigma$ denotes the shift functor in $\CC_{\widehat Q}$ and where $^\bot(\Sigma (k\widehat Q/kQ))$ denotes the full subcategory consisting of objects $M$ in $\CC_{\widehat Q}$ such that $\Ext^1_{\CC_{\widehat Q}}(M, P_i)=0$ for any $i \in Q_0'$. Thus objects in $\CC_Q$ can be identified with objects $M$ in $\CC_{\widehat Q}$ such that $\Ext^1_{\CC_{\widehat Q}}(M, P_i)=0$ for any $i \in Q_0'$ and such that $M \not \simeq P_i$ for any $i \in Q_0'$.
		      
		      Given an object $M$ in $\CC_Q$, the cluster character $X'_M \in \Z[\textbf y, \textbf x^{\pm 1}]$ associated to $M$ by Fu and Keller is defined as follows. Using the above equivalence of categories, $M$ is viewed as an object in $^\bot(\Sigma (k\widehat Q/kQ))/(k\widehat Q/kQ)$ and the character $X'_M$ is $\pi(X^{k\widehat Q}_{\Sigma M})$ where $X^{k\widehat Q}_?:\Ob(\CC_{\widehat Q}) \fl \Z[\textbf u^{\pm 1}]$ is the cluster character on $\CC_{\widehat Q}$ associated by Palu to the cluster-tilting object $k\widehat Q$ in $\CC_{\widehat{Q}}$ (see \cite{Palu} for details). 
		      
		      Fix thus an indecomposable object $M$ in $^\bot(\Sigma (k\widehat Q/kQ))/(k\widehat Q/kQ)$. If $M$ is not a projective $k\widehat Q$-module, then $\Sigma M$ is a $k\widehat Q$-module and
		      $$0=\Hom_{\CC_{\widehat Q}}(P_i, \Sigma M)=\Hom_{k{\widehat Q}}(P_i, \Sigma M)=\ddim M(i)$$
		      for any $i \in Q_0'$ so that $\Sigma M$ can be viewed as a representation of $Q$. In particular, there is some $kQ$-module $M_0$ such that $\Sigma(M)=\iota(M_0)$. Thus, we get equalities
		      $$X'_M=\pi(X^{k\widehat Q}_{\Sigma M})=\pi(X^{\widehat Q}_{\Sigma M})=\XQy_{\iota(M_0)}$$
		      where the second equality follows from \cite[Section 5]{Palu} and the last equality follows from Lemma \ref{lem:addingcoeffs}.
		      If $M$ is a projective module $P_j$ for some $j \in Q_0$, then
		      $$X'_M=\pi(X^{k\widehat Q}_{\Sigma P_j})=\pi(u_j)=x_j=\XQy_{P_j[1]}.$$
		      
		      Conversely, for any object $M$ in $kQ$-mod, $\iota(M)$ is an object in $k\widehat Q$-mod such that 
		      $$0=\Hom_{k{\widehat Q}}(P_i, \iota(M))=\Hom_{\CC_{\widehat Q}}(P_i, \iota(M))\textrm{ for any }i \in Q_0'$$
		      so that $\Sigma^{-1}\iota(M)$ belongs to $^\bot(\Sigma (k\widehat Q/kQ))/(k\widehat Q/kQ)$. Thus,
		      $$\XQy_{M}=\pi(X^{\widehat Q}_{\iota(M)})=\pi(X^{k\widehat Q}_{\iota(M)})=X'_{\Sigma^{-1}\iota(M)}$$
		      where the first equality follows from Lemma \ref{lem:addingcoeffs} and the second follows from \cite[Section 5]{Palu}. Thus, cluster characters with coefficients we defined coincide with those previously introduced by Fu and Keller. In particular, the cluster variables in $\mathcal A(Q,\textbf x, \textbf y)$ are precisely the characters $\XQy_M$ when $M$ runs over the indecomposable rigid objects in $\CC_Q$ \cite{FK}.
	      \end{rmq}
	\end{subsection}
\end{section}

\begin{section}{Characters with coefficients in Dynkin type $\A$}\label{section:typeA}
	Let $r \geq 1$ be an integer and $A$ denote the quiver of type $\Alin{r}$, that is, of Dynkin type $\A_r$ equipped with the following orientation~:
	$$\xymatrix{
		      0 & \ar[l] 1 & \ar[l] 2 & \ar[l] \cdots & \ar[l] r-1.
	}$$
	For any $i \in [0, r-1]$, $n \in [1, r-i]$, we denote by $S_i^{(n)}$ the unique (up to isomorphism) indecomposable $kA$-module with socle $S_i$ and length $n$. By convention, for any $i \in [0, r-1]$, $S_i^{(0)}$ denotes the zero module. For simplicity, we denote by $i+r$ the vertex $\sigma(i) \in A_0'$ for any $i \in [0,r-1]$.

	The following lemma is analogous to \cite[Lemma 4.2.1]{mathese}~:
	\begin{lem}\label{lem:multpsinA}
	      For any $i \in [0,r-2]$ and $n \in [1,r-1-i]$, the following holds~:
	      $$\XAy_{S_i^{(n)}}\XAy_{S_{i+1}^{(n)}}=\XAy_{S_i^{(n+1)}}\XAy_{S_{i+1}^{(n-1)}}+\textbf y^{\ddim S_{i+1}^{(n)}}.$$
	\end{lem}
	\begin{proof}
	      For any $i \in [0,r-2]$ and $n \in [1,r-1-i]$, there is an almost split sequence
	      $$0 \fl S_i^{(n)} \fl S_i^{(n+1)} \oplus S_{i+1}^{(n-1)} \fl S_{i+1}^{(n)} \fl 0.$$
	      The lemma is thus a direct consequence of Proposition \ref{prop:multps}.
	\end{proof}

	We now prove a relation analogous to three terms recurrence relations in the context of orthogonal polynomials. This relation will be essential in order to extract quantized Chebyshev polynomials.
	\begin{lem}\label{lem:recpolyinA}
	      For any $i \in [0,r-2]$ and $n \in [1,r-1-i]$, we have
	      $$\XAy_{S_i^{(n)}}\XAy_{S_{i+n}}=\XAy_{S_i^{(n+1)}}+y_{i+n}\XAy_{S_{i}^{(n-1)}}.$$
	\end{lem}
	\begin{proof}
	      Let $i \in [0,r-2]$ and $n \in [1,r-1-i]$. 
	      We consider the indecomposable $k\widehat{A}$-modules $\iota(S_i^{(n)})$ and $\iota(S_{i+n})$ in the cluster category $\mathcal C_{\widehat{A}}$. We thus have isomorphisms of vector spaces (see \cite{BMRRT})~:
	      \begin{align*}
		      \Ext^1_{\mathcal C_{\widehat{A}}}(\iota(S_{i+n}), \iota(S_i^{(n)})) 
			      &\simeq \Ext^1_{k{\widehat{A}}}(\iota(S_{i+n}), \iota(S_i^{(n)})) \oplus \Ext^1_{k{\widehat{A}}}(\iota(S_i^{(n)}), \iota(S_{i+n})) \\
			      &\simeq \Ext^1_{kA}(S_{i+n}, S_i^{(n)}) \oplus \Ext^1_{kA}(S_i^{(n)}, S_{i+n}) \\
			      &\simeq \Ext^1_{kA}(S_{i+n}, S_i^{(n)})\\
			      &\simeq \Hom_{kA}(S_i^{(n)}, S_{i+n-1})\\
			      &\simeq k
	      \end{align*}
	      So we can apply Caldero-Keller's one-dimensional multiplication formula for cluster characters without coefficients \cite{CK2} to $\iota(S_{i+n})$ and $\iota(S_i^{(n)})$ in $\mathcal C_{\widehat{A}}$. We get~:
	      $$X^{\widehat{A}}_{\iota(S_{i+n})}X^{\widehat{A}}_{\iota(S_i^{(n)})}=X^{\widehat{A}}_{\iota(S_i^{(n+1)})}+X^{\widehat{A}}_B$$
	      where $B=\ker \hat f \oplus \coker \hat f[-1]_{\widehat{A}}$ for any $0 \neq \hat f \in \Hom_{k\widehat{A}}(\iota(S_i^{(n)}), \tau_{\widehat{A}}(\iota(S_{i+n-1}))) \simeq k$.
	      
	      We now have to compute $\Hom_{k\widehat{A}}(\iota(S_i^{(n)}), \tau_{\widehat{A}}(\iota(S_{i+n-1})))$. For this, we first compute $\tau_{\widehat{A}}(\iota(S_{i+n-1}))$ taking care of the fact that $\iota$ does not commute with the Auslander-Reiten translation.
	      
	      In order to fix notations, we draw $\widehat{A}$ as follows~:
	      $$\xymatrix{
	      r & r+1  & \cdots  &  2r-2 & 2r-1\\
	      0 \ar[u] & 1 \ar[l] \ar[u] & \cdots \ar[l] & r-2 \ar[l] \ar[u] & r-1 \ar[l] \ar[u]
	      }$$
	      
	      We compute that a projective resolution of $S_{i+n-1}$ is given by
	      $$P_{i+n-1} \oplus P_{i+n+r} \xrightarrow{f} P_{i+n} \fl S_{i+n} \fl 0.$$
	      Applying the Nakayama functor $\nu$ we get
	      $$I_{i+n-1} \oplus I_{i+n+r} \xrightarrow{\nu(f)} I_{i+n}$$
	      where $\nu(f)$ is surjective since $I_{i+n-1} \fl I_{i+n}$ is onto. It follows from \cite{Gabriel:AR} that $\tau_{k\widehat{A}}(\iota(S_{i+n})) \simeq \ker \nu(f)$ and thus $\ker \nu(f)$ is the representation given by
	      $$\xymatrix{
	      0 & \cdots & 0 & 0 & k  & 0& \cdots & 0\\
	      0 \ar[u] & \cdots \ar[l] & 0 \ar[l] \ar[u] & k \ar[l] \ar[u] & k \ar[l] \ar[u]  & k \ar[l] \ar[u] & \cdots \ar[l]  & k \ar[l] \ar[u] \\
	      & & & i+n-1 \ar@{..}[u]
	      }$$
	      where the arrows are obviously zero or identity maps.
	      
	      Since $\iota(S_i^{(n)})$ is the representation given by 
	      $$\xymatrix{
	      0 & \cdots & 0 & 0 & \cdots & 0 & 0 & \cdots & 0\\
	      0 \ar[u] & \cdots \ar[l] & 0 \ar[u] \ar[l] & k \ar[l] \ar[u] & \cdots \ar[l] & k \ar[l] \ar[u] & 0 \ar[l] \ar[u]  & \cdots \ar[l] & 0 \ar[l] \ar[u] \\
	      & & & i \ar@{..}[u] & & i+n-1 \ar@{..}[u]
	      }$$
	      we get that for any non-zero morphism $\hat f$, the kernel $\ker \hat f$ is given by 
	      $$\xymatrix{
	      0 & \cdots & 0 & 0 & \cdots & 0 & 0 & \cdots & 0\\
	      0 \ar[u] & \cdots \ar[l] & 0 \ar[u] \ar[l] & k \ar[l] \ar[u] & \cdots \ar[l] & k \ar[l] \ar[u] & 0 \ar[l] \ar[u]  & \cdots \ar[l] & 0 \ar[l] \ar[u] \\
	      & & & i \ar@{..}[u] & & i+n-2 \ar@{..}[u]
	      }$$
	      which is isomorphic to $\iota(S_i^{(n-1)})$ and
	      $\coker \hat f$ is
	      $$\xymatrix{
	      0 & \cdots & 0 & k & 0 & \cdots & 0 & \cdots & 0\\
	      0 \ar[u] & \cdots \ar[l] & 0 \ar[u] \ar[l] & k \ar[l] \ar[u] & k \ar[l] \ar[u] & \cdots \ar[l] & k \ar[l] \ar[u]  & \cdots \ar[l] & k \ar[l] \ar[u] \\
	      & & & i+n \ar@{..}[u] & & 
	      }$$
	      which is isomorphic to the injective $k\widehat{A}$-module $I_{i+n+r}$. It thus follows that 
	      $$B \simeq \iota(S_i^{(n-1)}) \oplus P_{i+n+r}[1]$$
	      and
	      $$X^{\widehat{A}}_B=X^{\widehat{A}}_{\iota(S_i^{(n-1)})}u_{i+n+r}.$$
	      Thus, $$X^{\widehat{A}}_{\iota(S_{i+n})}X^{\widehat{A}}_{\iota(S_i^{(n)})}=X^{\widehat{A}}_{\iota(S_i^{(n+1)})}+u_{i+n+r}X^{\widehat{A}}_{\iota(S_i^{(n-1)})}.$$
	      Applying the homomorphism $\pi$ of Lemma \ref{lem:addingcoeffs} to this identity, we get 
	      $$\XAy_{S_{i+n}}\XAy_{S_i^{(n)}}=\XAy_{S_i^{(n+1)}}+y_{i+n}\XAy_{S_i^{(n-1)}}$$
	      and the lemma is proved.
	\end{proof}

	\begin{lem}\label{lem:simplealgind}
	      Let $A$ be a quiver of type $\Alin{r}$ with $r$ even. Then the set
	      $$\ens{\XAy_{S_i} |i \in [0,r-1]}$$
	      is algebraically independent over $\Z[\textbf y]$.
	\end{lem}
	\begin{proof}
	      Denote by $B$ the incidence matrix of $A$. As $r$ is even, $B$ is of full rank and thus there exists a $\Z$-linear form $\epsilon$ on $\Z^{Q_0}$ such that $\epsilon(B\alpha_i)<0$ for every $i \in [0,r-1]$. It thus follows from \cite{CK1} that 
	      $$F_n=\left(\bigoplus_{\epsilon(\nu) \leq n} \Z\prod_{i \in Q_0}x_i^{\nu_i}\right) \cap \Z[X^{A}_{S_i}|i \in [0,r-1]]$$
	      defines a filtration on $\Z[X^{A}_{S_i}|i \in [0,r-1]$ and in the associated graded algebra, we have
	      $$\textrm{gr}(X^A_{M})=\textrm{gr}\prod_{i=0}^{r-1}x_i^{-\<S_i,M\>}$$
	      for every $kA$-module $M$. We now consider the grading on $\Z[\XAy_{S_i}|i \in [0,r-1]]$ given the grading on $\Z[X^A_{S_i}|i \in [0,r-1]]$ and $\deg(y_i)=0$ for every $i \in Q_0$. We thus have that 
	      $$\textrm{gr}(\XAy_{M})=\textrm{gr}\prod_{i=0}^{r-1}x_i^{-\<S_i,M\>}$$
	      and thus, since $\left(\<S_i,M\>\right)_{i=0\ldots r-1} \neq \left(\<S_i,N\>\right)_{i=0\ldots r-1}$ if $\ddim M \neq \ddim N$, it follows that any finite set $\ens{\XAy_{M_i}|i \in J}$ with $\ddim M_i \neq \ddim M_j$ is linearly independent over $\Z[\textbf y]$. 
	      
	      Now assume that there is a polynomial $P(\textbf t)=\sum_{\nu \in \N^{[0,r-1]}} a_\nu t_0^{\nu_0}\cdots t_{r-1}^{\nu_{r-1}}$ such that 
	      $$P(\XAy_{S_0}, \ldots, \XAy_{S_{r-1}})=0.$$
	      
	      Since $(\XAy_{S_0})^{\nu_0} \cdots (\XAy_{S_{r-1}})^{\nu_{r-1}}=\XAy_{\bigoplus_{i=0}^{r-1}S_i^{\oplus \nu_i}}$, and $\ddim \left(\bigoplus_{i=0}^{r-1}S_i^{\oplus \nu_i} \right)=\nu$, we get a vanishing $\Z[\textbf y]$-linear combination of $\XAy_{M_\nu}$ where $\nu$ runs over a finite subset of $\N^{[0,r-1]}$. Since $\ddim M_\nu=\nu$, it follows from the above discussion that each of the $a_\nu$ is zero and thus the set $\ens{\XAy_{S_i} |i \in [0,r-1]}$ is algebraically independent over $\Z[\textbf y]$.
	\end{proof}
\end{section}

\begin{section}{Quantized Chebyshev polynomials}\label{section:quantized}
	\begin{subsection}{Quantized Chebyshev polynomials of infinite rank}
	      Let $\textbf q=\ens{q_i | i \in \Z}$ be a family of indeterminates over $\Z$ and $\ens{x_{i,1} | i \in \Z}$ be a family of indeterminates over $\Z[\textbf q]$. We define by induction a family $$\ens{x_{i,n} |i \in \Z, n \geq 1} \subset \Q(\textbf q)(x_{i,1} |i \in \Z)$$
	      by
	      $$x_{i,n}x_{i+1,n}=x_{i,n+1}x_{i+1,n-1}+\prod_{k=1}^n q_{i+k}$$
	      with the convention that $x_{i,0}=1$ for all $i \in \Z$. 
	      For simplicity if $I=[i,j] \subset \Z$ is an interval, we write 
	      $$\textbf x_{I}=(x_{i,1}, \ldots, x_{j,1})
	      \textrm{ and }
	      \textbf q_{I}=(q_{i}, \ldots, q_{j}).$$
	      It follows directly from the definition that for every $n$, there exists a rational function $P_{n}$ such that
	      $$x_{i,n}=P_{n}(\textbf q_{[i,i+n-1]}, \textbf x_{[i,i+n-1]}).$$
	      
	      \begin{prop}\label{prop:Pnpoly}
		      For every $i \in \Z$ and $n \geq 1$, $P_{n}$ is the polynomial given by
		      $$P_{n}(\textbf q_{[i,i+n-1]}, \textbf x_{[i,i+n-1]})=\det \left[\begin{array}{ccccccc}
		      x_{i+n-1,1} & 1 &&& (0)\\
		      q_{i+n-1} & \ddots & \ddots \\
		      & \ddots & \ddots & \ddots \\
		      & & \ddots & \ddots & 1 \\
		      (0)& & & q_{i+1} & x_{i,1}
	      \end{array}\right].$$
	      \end{prop}
	      \begin{proof}
		      Let $i \in \Z$ and $n \geq 1$, fix an even integer $r > i+n$. Let $A$ still denote the quiver of section \ref{section:typeA} and denote by $\XAy_?$ the associated cluster character with coefficients. Consider the homomorphism of $\Z$-algebras~:
		      $$\phi:\left\{\begin{array}{rcll}
			      \Z[\XAy_{S_i} |i \in [0,r-1]] & \fl & \Z[\textbf q][\textbf x_{[i,i+n-1]}]\\
			      y_i & \mapsto & q_i & \textrm{ for all } i \in [0,r-1]\\
			      \XAy_{S_i} & \mapsto & x_{i,1} & \textrm{ for all } i \in [0,r-1]\\
		      \end{array}\right.$$
		      By Lemma \ref{lem:simplealgind}, $\phi$ is an isomorphism.
		      By Lemma \ref{lem:multpsinA}, for any $j \in [i, i+n-1]$ and $k <n-i$ we have
		      $$\XAy_{S_j^{(k)}}\XAy_{S_{j+1}^{(k)}}=\XAy_{S_j^{(k+1)}}\XAy_{S_{j+1}^{(k-1)}}+\textbf y^{\ddim S_{j+1}^{(k)}}$$
		      and $\textbf y^{\ddim S_{j+1}^{(k)}}=\prod_{l=1}^k y_{j+l}$. Since the $x_{j,k}$ for $ 1 \leq k \leq n$ are obtained by $$x_{j,k}x_{j+1,k}=x_{j,k+1}x_{j+1,k-1}+\prod_{l=1}^k q_{j+l}$$
		      an immediate induction proves that 
		      $$\phi(\XAy_{S_j^{(k)}})=x_{j,k}$$
		      for any $j \in [i, i+n-1]$ and $k <n-i$. In particular, by Lemma \ref{lem:recpolyinA} that 
		      $$\XAy_{S_i^{(n)}}\XAy_{S_{i+n}}=\XAy_{S_i^{(n+1)}}+y_{i+n}\XAy_{S_{i}^{(n-1)}}$$
		      and thus applying $\phi$ we get
		      $$x_{i,n}x_{i+n,1}=x_{i,n+1}+q_{i+n}x_{i,n-1}$$
		      and thus by induction 
		      $$x_{i,n}=\det \left[\begin{array}{ccccccc}
			      x_{i+n-1,1} & 1 &&& (0)\\
			      q_{i+n-1} & \ddots & \ddots \\
			      & \ddots & \ddots & \ddots \\
			      & & \ddots & \ddots & 1 \\
			      (0)& & & q_{i+1} & x_{i,1}
		      \end{array}\right]$$ is a polynomial in $\Z[\textbf q_{[i,i+n-1]},\textbf x_{[i,i+n-1]}]$
		      and the proposition is proved.
	      \end{proof}
	      
	      As an immediate corollary, quantized Chebyshev polynomials are characterized by the following three-terms recurrence relation~:
	      \begin{corol}\label{corol:3terms}
			For any $n \geq 2$, the following equality holds~:
			$$P_{n+1}(\textbf q_{[i,i+n]},\textbf x_{[i,i+n]})=x_{i+n}P_n(\textbf q_{[i,i+n-1]},\textbf x_{[i,i+n-1]})-q_{i+n}P_{n-1}(\textbf q_{[i,i+n-2]}, \textbf x_{[i,i+n-2]})$$
	      \end{corol}

	      \begin{defi}
		      For any $n \geq 1$, $P_{n}$ is called the \emph{$n$-th quantized Chebyshev polynomial of infinite rank}. By convention $P_{0}=1$.
	      \end{defi}
	      
	      \begin{exmp}\label{exmp:Chebinfinite}
		      The first quantized Chebyshev polynomials of infinite rank are~:
		      $$\begin{array}{|r|l|}
			      \hline
			      P_{1}(q_{0},t_0) =& t_0\\
			      \hline
			      P_{2}(\textbf q_{[0,1]},\textbf t_{[0,1]}) =& t_0t_1-q_1\\
			      \hline
			      P_{3}(\textbf q_{[0,2]},\textbf t_{[0,2]}) =& t_0t_1t_2 - q_2 t_0 - q_1t_2\\
			      \hline
			      P_{4}(\textbf q_{[0,3]},\textbf t_{[0,3]}) =&
				      t_0t_1t_2t_3 - q_3t_0t_1 -q_1 t_2 t_3 - q_2 t_0 t_3 + q_1 q_3\\
			      \hline
			      P_{5}(\textbf q_{[0,4]},\textbf t_{[0,4]}) =&   t_0 t_1 t_2 t_3 t_4  - q_1 t_2 t_3 t_4 - q_2 t_0 t_3 t_4 - q_4 t_0 t_1 t_2\\  &- q_3 t_0 t_1 t_4 + q_1 q_4 t_2 + q_2 q_4 t_0 + q_1 q_3 t_4 \\ \hline
		      \end{array}$$
	      \end{exmp}
	      
		We now prove Theorem 3. 
		\begin{corol}
			Let $r \geq 1$ be an integer and $A$ be the quiver of type $\Alin{r}$ equipped with the following orientation
			$$\xymatrix{
				0 & \ar[l] 1 & \ar[l] 2 & \ar[l] \cdots & \ar[l] r-1.
			}$$
			Then, for any $i \in [0,r-1]$ and $n \in [1,r-i]$, we have
			$$
				X^{A,\textbf y}_{S_i^{(n)}}=P_n(y_i, \ldots, y_{i+n-1}, X^{A,\textbf y}_{S_i}, \ldots, X^{A,\textbf y}_{S_{i+n-1}})
			$$
			or equivalently
			$$X^{A,\textbf y}_{S_i^{(n)}}=\det \left[\begin{array}{ccccccc}
			      X^{A,\textbf y}_{S_{i+n-1}} & 1 &&& (0)\\
			      y_{i+n-1} & \ddots & \ddots \\
			      & \ddots & \ddots & \ddots \\
			      & & \ddots & \ddots & 1 \\
			      (0)& & & y_{i+1} & X^{A,\textbf y}_{S_i}
		      \end{array}\right].$$
		\end{corol}
		\begin{proof}
			Consider the epimorphism of $\Z$-algebras
			$$\pi:\left\{\begin{array}{rcll}
				\Z[q_i,x_{i,1}| i \in \Z] & \fl & \Z[y_i,\XAy_{S_i} | i \in [0,r-1]] & \\
				x_{i,1} & \mapsto & X_{S_i} & \textrm{ for } i \in [0,r-1]\\
				x_{i,1} & \mapsto & 1 & \textrm{ for } i \not \in [0,r-1]\\
				q_i & \mapsto & y_i &  \textrm{ for } i \in [0,r-1]\\
				q_i & \mapsto & 1 & \textrm{ for } i \not \in [0,r-1]
			\end{array}\right.$$

			By Lemma \ref{lem:recpolyinA}, for any $i \in [0,r-1]$ and $n \in [1,r-1-i]$, we have $$\XAy_{S_i^{(n)}}\XAy_{S_{i+n}}=\XAy_{S_i^{(n+1)}}+y_{i+n}\XAy_{S_{i}^{(n-1)}}$$
			so that $\pi(x_{i,n})=X_{S_i^{(n)}}$ for any $i \in [0,r-1]$ and $n \in [1,r-i]$. The result thus follows from Proposition \ref{prop:Pnpoly}.
		\end{proof}
	\end{subsection}

	\begin{subsection}{Quantized Chebyshev polynomials of finite ranks}
	      Fix now an integer $p \geq 1$ and consider the \emph{$n$-th quantized Chebyshev polynomial of rank $p$} $P_{n,p}$ defined in the introduction.
	      
	      \begin{exmp}
		      The first five quantized Chebyshev polynomials of rank 1 are
		      $$\begin{array}{|r|l|}
			      \hline
			      P_{1,1}(q,t) =& t\\
			      \hline
			      P_{2,1}(q,t) =& t^2-q\\
			      \hline
			      P_{3,1}(q,t) =& t^3 - 2 qt\\
			      \hline
			      P_{4,1}(q,t) =&
				      t^4-3qt^2+q^2\\
			      \hline
			      P_{5,1}(q,t) =&   t^5  - 4 qt^3 + 3q^2t \\ \hline
		      \end{array}$$
		      
		      The first five quantized Chebyshev polynomials of rank 2 are
		      $$\begin{array}{|r|l|}
			      \hline
			      P_{1,2}(q_0,q_1,t_0,t_1) =& t_0\\
			      \hline
			      P_{2,2}(q_0,q_1,t_0,t_1) =& t_0t_1-q_1\\
			      \hline
			      P_{3,2}(q_0,q_1,t_0,t_1) =& t_0^2t_1 - q_0 t_0 - q_1t_0\\
			      \hline
			      P_{4,2}(q_0,q_1,t_0,t_1) =&
				      t_0^2t_1^2 - q_1t_0t_1 - q_1 t_0 t_1 - q_2 t_0 t_1 + q_1^2\\
			      \hline
			      P_{5,2}(q_0,q_1,t_0,t_1) =&  t_0^3 t_1^2 - 2 q_1 t_0^2 t_1 - 2 q_0 t_0^2 t_1 + q_0 q_1 t_0 + q_0^2 t_0 + q_1^2 t_0\\
			      \hline
		      \end{array}$$
		      
		      The first five quantized Chebyshev polynomials of rank 3 are
		      $$\begin{array}{|r|l|}
			      \hline
			      P_{1,3}(\textbf q_{[0,2]}, \textbf t_{[0,2]}) =& t_0\\
			      \hline
			      P_{2,3}(\textbf q_{[0,2]}, \textbf t_{[0,2]}) =& t_0t_1-q_1\\
			      \hline
			      P_{3,3}(\textbf q_{[0,2]}, \textbf t_{[0,2]}) =& t_0t_1t_2 - q_2 t_0 - q_1t_2\\
			      \hline
			      P_{4,3}(\textbf q_{[0,2]}, \textbf t_{[0,2]}) =&
				      t_0^2t_1t_2-q_1t_0t_2-q_0t_0t_1-q_2t_0^2+q_0q_1\\
			      \hline
			      P_{5,3}(\textbf q_{[0,2]},\textbf t_{[0,2]}) =&  
				      t_0^2t_1^2t_2-2q_1t_0t_1t_2-q_2t_0^2t_1+q_1^2t_2-q_0t_0t_1^2+q_1q_2t_0+q_0q_1t_1\\
			      \hline
		      \end{array}$$
		      
		      The first five quantized Chebyshev polynomials of rank 4 are
		      $$\begin{array}{|r|l|}
			      \hline
			      P_{1,4}(\textbf q_{[0,3]}, \textbf t_{[0,3]}) =& t_0\\
			      \hline
			      P_{2,4}(\textbf q_{[0,3]}, \textbf t_{[0,3]}) =& t_0t_1-q_1\\
			      \hline
			      P_{3,4}(\textbf q_{[0,3]}, \textbf t_{[0,3]}) =& t_0t_1t_2 - q_2 t_0 - q_1t_2\\
			      \hline
			      P_{4,4}(\textbf q_{[0,3]}, \textbf t_{[0,3]}) =&
				      t_0t_1t_2t_3 - q_3t_0t_1 -q_1 t_2 t_3 - q_2 t_0 t_3 + q_1 q_3\\
			      \hline
			      P_{5,4}(\textbf q_{[0,3]}, \textbf t_{[0,3]}) =&   t_0^2 t_1 t_2 t_3 - q_1 t_0t_2 t_3 - q_2 t_0^2 t_3 - q_0 t_0 t_1 t_2\\  &- q_3 t_0^2 t_1 + q_0 q_1 t_2 + q_0 q_2 t_0 + q_1 q_3 t_0 \\ \hline
		      \end{array}$$
		      
		      The first five quantized Chebyshev polynomials of rank $p \geq 5$ coincide with the first five quantized Chebyshev polynomials of infinite rank.
	      \end{exmp}

	      Note that quantized Chebyshev polynomials are deformations of generalized Chebyshev polynomials introduced in \cite{Dupont:stabletubes}, more precisely we have the following relation~:
	      \begin{lem}\label{lem:quantizeddeform}
		      For any $p \in \Z_{>0} \sqcup \ens \infty$, the $n$-th generalized Chebyshev polynomial of rank $p$ coincides with the $n$-th quantized Chebyshev polynomial of rank $p$ where all the $q_i$'s are specialized at $1$.
	      \end{lem}
	      \begin{proof}
		      We recall the construction of generalized Chebyshev polynomials given in \cite{Dupont:stabletubes}. Fix $p \geq 0$ an integer and set $\ens{a_{i,1}|i \in \Z/p\Z}$ a family of indeterminates over $\Z$. Then the $n$-th Chebyshev polynomial of rank $p$ (resp. infinite rank) if $p>0$ (resp. $p=0$) is the expression of $a_{i,n}$ in terms of $\ens{a_{i,1}|i \in \Z/p\Z}$ where the $a_{i,n}$ for $n \geq 1$ are defined inductively by 
		      $$a_{i,n}a_{i+1,n}=a_{i,n+1}a_{i+1,n-1}+1.$$
		      It thus follows from the definition that, specializing all the $q_i$'s at $1$, the $n$-th generalized Chebyshev polynomial of rank $p$ (resp. infinite rank) is the specialization of the $n$-th quantized Chebyshev polynomial of rank $p$ (resp. infinite rank).
	      \end{proof}

	      It is proved in \cite{Dupont:stabletubes} that for every $n \geq 1$, the $n$-th generalized Chebyshev polynomial of rank 1 is the usual $n$-th normalized Chebyshev polynomial of the second kind. It thus follows from Lemma \ref{lem:quantizeddeform} that the $n$-th quantized Chebyshev polynomial of rank 1 specialized at $q_i=1$ for every $i \in \Z$ is the $n$-th normalized Chebyshev polynomial of the second kind. This case being of particular interest in the sequel, we set the following definition~:
	      
	      \begin{defi}\label{defi:nq2d}
		      The $n$-th quantized Chebyshev polynomial of rank 1 is called the $n$-th \emph{quantized Chebyshev polynomial of the second kind}.
	      \end{defi}
	\end{subsection}
\end{section}

\begin{section}{Quantized Chebyshev polynomials and characters for regular modules}\label{section:polyreg}
	As we already saw in Theorem 3, quantized Chebyshev polynomials appear in character formulas with coefficients associated to indecomposable modules over the path algebra of an equioriented quiver of Dynkin type $\A$. In this section, we prove that these polynomials also arise in character formulas with coefficients for indecomposable regular modules over the path algebra of an acyclic quiver of infinite representation type.

	Let $Q$ be an acyclic quiver of infinite representation type, denote by $\mathcal R$ a regular component in the Auslander-Reiten quiver $\Gamma(kQ)$ of $kQ$-mod. Let $p \geq 0$ such that $\mathcal R$ is is of the form $\Z\A_\infty/(p)$. If $Q$ is affine, then $p \geq 1$ and $\mathcal R$ is called a \emph{tube} \cite[Section 3.6]{ringel1099}. If $p=1$, $\mathcal R$ is called \emph{homogeneous} and if $p>1$, $\mathcal R$ is called \emph{exceptional}. If $Q$ is wild, then $p=0$ \cite{ringel:wild}. 

	We denote by $\ens{R_i, i \in \Z/p\Z}$ the set of quasi-simple modules in $\mathcal R$ ordered such that $\tau R_i \simeq R_{i-1}$ for every $i \in \Z/p\Z$. For any $i \in \Z/p \Z$ and $n \geq 1$, we denote by $R_i^{(n)}$ the unique indecomposable $kQ$-module with quasi-length $n$ and quasi-socle $R_i$. By convention, $R_i^{(0)}$ denotes the zero module for every $i \in \Z/p\Z$. With these notations, for any $n \geq 1$ and any $i \in \Z/p\Z$, there is an almost split exact sequence
	\begin{equation}\label{eq:suiteARinT}
	      0 \fl R_i^{(n)} \fl R_i^{(n+1)} \oplus R_{i+1}^{(n-1)} \fl R_{i+1}^{(n)} \fl 0.
	\end{equation}
	Locally, a regular component can be depicted as in Figure \ref{figure:regularcomp}.
	
	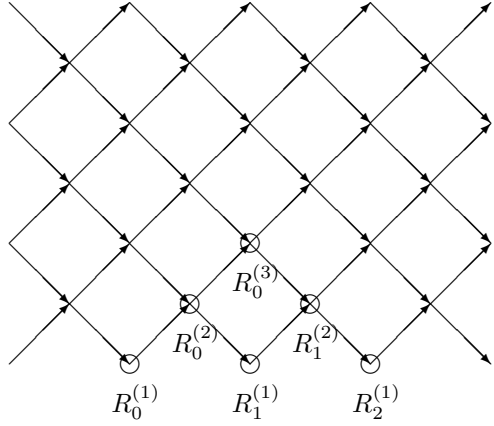
\begin{figure}[htb]
	      \setlength{\unitlength}{.8mm}
	      \begin{picture}(120,80)(-22,-10)
		      \multiput(0,0)(20,0){4}{\multiput(0,0)(0,20){3}{\vector(1,1){10}}}
		      \multiput(10,10)(20,0){4}{\multiput(0,0)(0,20){3}{\vector(1,-1){10}}}
		      \multiput(0,0)(20,0){4}{\multiput(0,20)(0,20){3}{\vector(1,-1){10}}}
		      \multiput(10,10)(20,0){4}{\multiput(0,0)(0,20){3}{\vector(1,1){10}}}
		      
		      \put(20,0){\circle{3}}
		      \put(17,-8){$R_0^{(1)}$}
		      
		      \put(30,10){\circle{3}}
		      \put(27,2){$R_0^{(2)}$}
		      
		      \put(40,20){\circle{3}}
		      \put(37,12){$R_0^{(3)}$}
		      
		      \put(40,0){\circle{3}}
		      \put(37,-8){$R_1^{(1)}$}
		      
		      \put(50,10){\circle{3}}
		      \put(47,2){$R_1^{(2)}$}
		      
		      \put(60,0){\circle{3}}
		      \put(57,-8){$R_2^{(1)}$}
		      
	      \end{picture}
	      \caption{Local configuration in a regular component of $\Gamma(kQ)$}\label{figure:regularcomp}
	\end{figure}

	We are now able to prove Theorem 2.
	\begin{theorem}\label{theorem:polycomposantesreg}
	      Let $Q$ be a quiver of infinite representation type, $\XQy_?$ be the cluster character with coefficients on $kQ$-mod. Let $\mathcal R$ be a regular component in $\Gamma(kQ)$ of the form $\Z\A_\infty/(p)$ for some $p \geq 0$. Let  $\ens{R_i|i \in \Z/p\Z}$ denote the quasi-simples of $\mathcal R$ ordered such that $\tau R_i \simeq R_{i-1}$ for every $i \in \Z/p\Z$. Then, for every $i \in \Z/p\Z$ and any $n \geq 1$, we have;
	      $$\XQy_{R_i^{(n)}}=P_{n}(\textbf y^{\ddim R_i}, \ldots, \textbf y^{\ddim R_{i+n-1}}, \XQy_{R_i}, \ldots, \XQy_{R_{i+n-1}}).$$
	      Moreover, if $p \geq 1$, then
	      $$\XQy_{R_i^{(n)}}=P_{n,p}(\textbf y^{\ddim R_i}, \ldots, \textbf y^{\ddim R_{i+p-1}}, \XQy_{R_i}, \ldots, \XQy_{R_{i+p-1}}).$$
	\end{theorem}
	\begin{proof}
	      We consider the $\Z$-families $\textbf q=\ens{q_i|i \in \Z}$ and $\ens{x_{i,1}|i \in \Z}$. Define $\phi$ to be the homomorphism of $\Z$-algebras 
	      $$\phi: \left\{\begin{array}{rcl}
		      \Z[q_i|i \in \Z][x_{i,1}|i \in \Z] & \fl & \Z[\textbf y][\textbf x^{\pm 1}]\\
		      q_i & \fl & \textbf y^{\ddim R_{i}}\\
		      x_{i,1} & \fl & \XQy_{R_{i}}\\
	      \end{array}\right.$$
	      where the indices on the right hand side are taken in $\Z/p\Z$.
	      We claim that $\phi(x_{i,n})=\XQy_{R_{i}^{(n)}}$ for every $i$ and $n \geq 0$. If $n=0,1$, the result holds by definition of $\phi$. Now, assume that $n \geq 1$, it follows from the almost split exact sequence (\ref{eq:suiteARinT}) and Proposition \ref{prop:multps} that for any $i \in \Z/p\Z$ we have
	      $$\XQy_{R_i^{(n)}} \XQy_{R_{i+1}^{(n)}} = \XQy_{R_{i}^{(n+1)}} \XQy_{R_{i+1}^{(n-1)}}+\textbf y^{\ddim R_{i+1}^{(n)}}$$
	      so that applying $\phi$ we get 
	      $$\phi(x_{i,n+1})=\XQy_{R_{i}^{(n+1)}}$$
	      for every $i \in \Z$. This proves the claim. Thus for every $i \in \Z/p\Z$ and every $n \geq 1$, we have
	      \begin{align*}
		      \XQy_{R_i^{(n)}} 
			      &=\phi(x_{i,n})\\
			      &=\phi(P_{n}(q_i, \ldots, q_{i+p-1}, x_{i,1}, \ldots, x_{i+n-1}))\\
			      &=P_{n}(\textbf y^{\ddim R_i}, \ldots, \textbf y^{\ddim R_{i+n-1}}, \XQy_{R_i}, \ldots, \XQy_{R_{i+n-1}})
	      \end{align*}
	      which proves the first assertion. For the second one, assume that $p \geq 1$, then $X_{R_{i+p}}=X_{R_{i}}$ and $\ddim R_{i+p}=\ddim R_i$ for every $i \in \Z$. It thus follows from the definition of $P_{n,p}$ that 
	      $$\XQy_{R_i^{(n)}}=P_{n,p}(\textbf y^{\ddim R_i}, \ldots, \textbf y^{\ddim R_{i+p-1}}, \XQy_{R_i}, \ldots, \XQy_{R_{i+p-1}})$$
	      and the theorem is proved.
	\end{proof}

	By expanding with respect to the first column in the determinantal expression of $P_n$, we prove the following immediate corollary~:
	\begin{corol}
	      With notations of Theorem \ref{theorem:polycomposantesreg} we have
	      $$\XQy_{R_i^{(n)}}\XQy_{R_{i+n}}=\XQy_{R_i^{(n+1)}}+\textbf y^{\ddim R_{i+n}}\XQy_{R_{i}^{(n-1)}}.$$
	      for any $n \geq 1$ and $i \in \Z/p\Z$.
	\end{corol}

	\begin{exmp}\label{exmp:A21}
	      We consider the following quiver of type $\Aaffine_{2,1}$~:
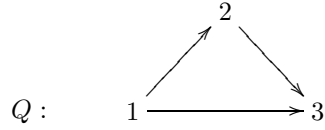
	      $$\xymatrix{
			      && 2\ar[rd]\\
		      Q: & 1 \ar[rr]\ar[ru] && 3
	      }$$
	      Cluster algebras with coefficients associated to this quiver are extensively studied in \cite{Cerulli:A21}. This example actually motivated the introduction of quantized Chebyshev polynomials.

	      $Q$ is an affine quiver. All but one regular components of $\Gamma(kQ)$ are homogeneous tubes, we denote by $\mathcal T_0$ the unique exceptional tube in $\Gamma(kQ)$. The quasi-simple modules in $\mathcal T_0$ are 
	      $$
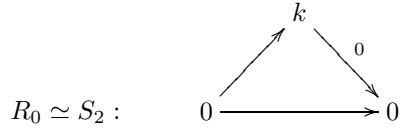
\xymatrix{
		      && k\ar[rd]^0\\
		      R_0 \simeq S_2: & 0 \ar[rr]\ar[ru] && 0
	      }$$ 
	      and 
	      
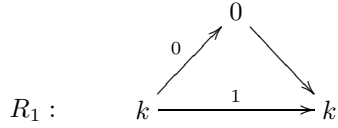
$$\xymatrix{
		      && 0\ar[rd]\\
		      R_1: & k \ar[rr]^1 \ar[ru]^0 && k
	      }$$
	      A direct computation shows that 
	      $$w(\textbf y)=\XQy_{R_0}=\frac{x_1+y_2x_3}{x_2} \textrm{ and } z(\textbf y)=\XQy_{R_1}=\frac{x_1x_2+y_3+y_1y_2y_3x_3}{x_1x_3},$$
	      $$w=X^Q_{R_0}=\frac{x_1+x_3}{x_2} \textrm{ and } z=X^Q_{R_1}=\frac{x_1x_2+1+x_3}{x_1x_3}.$$
	      
	      For any $\lambda \in \P^1(k)\setminus \ens {0}$, we denote by $M_\lambda$ the unique quasi-simple module in $\mathcal T_\lambda$. It is given by
	      $$
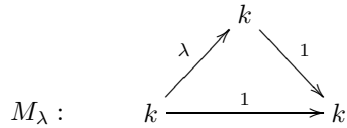
\xymatrix{
		      && k\ar[rd]^1\\
		      M_\lambda : & k \ar[rr]^1 \ar[ru]^\lambda && k
	      }$$
	      if $\lambda \neq \infty$ and 
	      
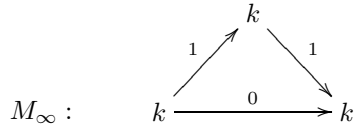
$$\xymatrix{
		      && k\ar[rd]^1\\
		      M_\infty : & k \ar[rr]^0 \ar[ru]^1 && k
	      }$$
	      A direct check proves that $\XQy_{M_\lambda}$ and $X^Q_{M_\lambda}$ do not depend on the choice of $\lambda \in \P^1(k)\setminus \ens 0$. We set
	      $$u(\textbf y)=\XQy_{M_\lambda}=\frac{x_1^2x_2+y_3x_1+y_2y_3x_3 + y_1 y_2 y_3 x_2 x_3^2}{x_1x_2x_3},$$
	      $$u=X^Q_{M_\lambda}=\frac{x_1^2x_2+x_1+x_3+x_2x_3^2}{x_1x_2x_3}$$
	      for any $\lambda \in \P^1(k) \setminus \ens 0$.
	      
	      A direct computation proves that 
	      $$\XQy_{R_0^{(2)}}=\frac{x_1^2x_2 + y_3x_1 + y_2x_1x_2x_3 + y_2y_3x_3 + y_1y_2y_3x_2x_3^2}{x_1x_2x_3}$$
	      and we easily check that
	      \begin{align*}
		      \XQy_{R_0^{(2)}}
			&=\XQy_{R_0}\XQy_{R_1}-y_1y_3\\
			&=P_{2,2}(\textbf y^{\ddim R_0}, \textbf y^{\ddim R_1}, \XQy_{R_0}, \XQy_{R_1}) \\
			&=P_{2,2}(y_2, y_1y_3, w(\textbf y), z(\textbf y)).
	      \end{align*}
	      
	      Similarly, a direct computation of the cluster character proves that
	      \begin{align*}
		      \XQy_{R_0^{(3)}}
			      &=\frac{1}{x_1x_2^2x_3}\left(x_1^3 x_2 + y_2 x_1^2 x_2 x_3 + y_3 x_1^2 + 2 y_2 y_3 x_1x_3 + \right.\\
			      &+\left. y_1 y_2 y_3 x_1 x_2 x_3^2 + y_2^2 y_3 x_3^2 +y_1 y_2^2 y_3 x_2 x_3^3 \right)
	      \end{align*}
	      and one verifies that
	      \begin{align*}
		      \XQy_{R_0^{(3)}} 
			      &=\XQy_{R_1}(\XQy_{R_0})^2+y_1y_3\XQy_{R_0}+y_2\XQy_{R_0}\\
			      &=\XQy_{R_1} (\XQy_{R_0})^2+\textbf y^{\ddim R_1}\XQy_{R_0}+\textbf y^{\ddim R_2}\XQy_{R_0}\\
			      &=P_{3,2}(\textbf y^{\ddim R_0}, \textbf y^{\ddim R_1}, \XQy_{R_0}, \XQy_{R_1})\\
			      &=P_{3,2}(y_2, y_1y_3, w(\textbf y), z(\textbf y)).
	      \end{align*}
	      
	      In homogeneous tubes, we can also compute the cluster character
	      \begin{align*}
		      \XQy_{M_\lambda^{(2)}}
			      &=\frac{1}{x_1^2x_2^2x_3^2}\left(x_1^4 x_2^2 + 2 y_3 x_1^3 x_2 + y_3^2 x_1^2 + 2 y_2 y_3 x_1^2 x_2 x_3 + 2 y_2 y_3^2 x_3 x_1 \right.\\ \\
			      &+\left. y_2^2 y_3^2 x_3^2 + y_1 y_2 y_3 x_1^2 x_2^2 x_3^2 + 2 y_1 y_2^2 y_3^2 x_2 x_3^3 + 2 y_1 y_2 y_3^2 x_1 x_2 x_3^2 + y_1^2 y_2^2 y_3^2 x_2^2 x_3^4 \right)\\
			      &=(\XQy_{M_\lambda})^2-\textbf y^{\ddim M_\lambda}\\
			      &=(\XQy_{M_\lambda})^2-\textbf y^{\delta}\\
			      &=P_{2,1}(\textbf y^{\delta}, \XQy_{M_\lambda})
	      \end{align*}
	      illustrating Theorem \ref{theorem:polycomposantesreg}.
	\end{exmp}
\end{section}

\begin{section}{Quantized Chebyshev polynomials of the first and second kinds}\label{section:firstandsecond}
	Normalized Chebyshev polynomials of the first kind initially appeared in the context of cluster algebras in \cite{shermanz}. They were introduced in order to study canonically positive bases in rank two cluster algebras. These polynomials are defined by three terms recurrence relations. Let $x$ be an indeterminate over $\Z$, then $F_n$ is the polynomial in one variable defined by 
	$$F_0(x)=1, \quad F_1(x)=x, \quad F_2(x)=x^2-2 \textrm{ and }$$
	$$F_{n+1}(x)=F_n(x)F_1(x)-F_{n-1}(x), \textrm{ for }n \geq 2.$$
	It is easy to check that these polynomials are characterized by
	$$F_n(x+x^{-1})=x^n+x^{-n}$$
	for every $n \geq 1$.

	Normalized Chebyshev polynomials of the second kind appeared in \cite{CZ} in order to study bases in the cluster algebra associated to the Kronecker quiver. These polynomials are defined be the following three terms recurrence relation. Let $x$ be an indeterminate over $\Z$, then $S_n$ is the polynomial in one variable defined by 
	$$S_0(x)=1, \quad S_1(x)=x, \quad S_2(x)=x^2-1 \textrm{ and }$$
	$$S_{n+1}(x)=S_n(x)S_1(x)-S_{n-1}(x), \textrm{ for }n \geq 2.$$
	It is easy to check that these polynomials are characterized by
	$$S_n(x+x^{-1})=\sum_{k=0}^n x^{n-2k}$$
	for every $n \geq 1$.

	In particular, it appears that normalized Chebyshev polynomials of the first and second kind satisfy the same recurrence relations but second terms differ. We now prove that there is a similar phenomenon for quantized Chebyshev polynomials. We define the quantized Chebyshev polynomials of the first kind as follows. Let $\nu$ be an indeterminate over $\Z$, $q=\nu^2$ and $x$ be an indeterminate over $\Z[q]$. 

	\begin{defi}\label{defi:firstkind}
	      The \emph{$n$-th quantized Chebyshev polynomial of the first kind} is the polynomial $F_n^q(x) \in \Z[q][x]$ defined by
	      $$F_0^q(x)=1,\quad F_1^q(x)=x,\quad F_2^q(x)=x^2-2q \textrm{ and }$$
	      $$F_{n+1}^q(x)=F_n^q(x)F_1^q(x)-qF_{n-1}^q(x) \textrm{ for }n \geq 2.$$
	\end{defi}
	\begin{exmp}
	      The first five quantized Chebyshev polynomials of the first kind are given by
	      $$\begin{array}{|r|l|}
		      \hline
		      F_1^q(x) =& x\\
		      \hline
		      F_2^q(x) =& x^2-2q\\
		      \hline
		      F_3^q(x) =& x^3 - 3 qx\\
		      \hline
		      F_4^q(x) =& x^4-4qx^2+2q^2\\
		      \hline
		      F_5^q(x) =& x^5-5qx^3+5q^2x \\ \hline
	      \end{array}$$	
	\end{exmp}

	From now on, we will denote by $S_n^q(x)=P_{n,1}(q,x)$ the $n$-th quantized Chebyshev polynomial of the second kind introduced in definition \ref{defi:nq2d}. It follows from Corollary \ref{corol:3terms} that these polynomials are characterized by 
	$$S_0^q(x)=1,\quad S_1^q(x)=x,\quad S_2^q(x)=x^2-q \textrm{ and }$$
	$$S_{n+1}^q(x)=S_n^q(x)S_1^q(x)-qS_{n-1}^q(x) \textrm{ for }n \geq 2.$$
	Thus, as in the non-quantized case, first kind and second kind quantized Chebyshev polynomials satisfy the same induction relations but second terms differ. 

	As in the non-quantized case, we now give algebraic characterizations of quantized Chebyshev polynomials of first and second kinds.
	\begin{lem}\label{lem:caracalgqCheb}
	      Let $t$ be an indeterminate over $\Z[q]$, then for any $n \geq 1$, we have
	      \begin{enumerate}
		      \item $$F_n^q({\nu}(t+t^{-1}))={\nu}^n(t^n+t^{-n}),$$
		      \item $$S_n^q({\nu}(t+t^{-1}))={\nu}^n\sum_{k=0}^n t^{n-2k}.$$
	      \end{enumerate}
	\end{lem}
	\begin{proof}
	      We first prove the property for quantized Chebyshev polynomials of the first kind. We prove it by induction on $n$. The property holds for $n=1$. We have~:
	      \begin{align*}
		      F_{n+1}^q({\nu}t+{\nu}t^{-1}) 
			      &= F_n^q({\nu}t+{\nu}t^{-1})F_1^q({\nu}t+{\nu}t^{-1})-qF_{n-1}^q({\nu}t+{\nu}t^{-1})\\
			      &= ({\nu}^nt^n+{\nu}^nt^{-n})({\nu}t+{\nu}t^{-1})-{\nu}^2({\nu}^{n-1}t^{n-1}+{\nu}^{n-1}t^{1-n})\\
			      &= {\nu}^{n+1}t^{n+1}+{\nu}^{n+1}t^{-(n+1)}.
	      \end{align*}
	      Now for quantized Chebyshev polynomials of the second kind we have~:
	      \begin{align*}
		      S_{n+1}^q({\nu}t+{\nu}t^{-1}) 
			      &= S_n^q({\nu}t+{\nu}t^{-1})S_1^q({\nu}t+{\nu}t^{-1})-qS_{n-1}^q({\nu}t+{\nu}t^{-1})\\
			      &= {\nu}^n(\sum_{k=0}^n t^{n-2k})({\nu}t+{\nu}t^{-1})-{\nu}^2{\nu}^{n-1}(\sum_{k=0}^{n-1} t^{n-1-2k})\\
			      &= {\nu}^n(\sum_{k=0}^n t^{n-2k})({\nu}t+{\nu}t^{-1})-{\nu}^{n}(\sum_{k=0}^{n-1} t^{n-2k})\\
			      &= {\nu}^{n+1}(\sum_{k=0}^{n+1} t^{n+1-2k}).
	      \end{align*}
	\end{proof}

	As a corollary, we obtain~:
	\begin{corol}\label{corol:secondasfirst}
	      For any $n \geq 0$, we have
	      $$S^q_n(x)=\sum_{k=0}^n q^kF^q_{n-2k}$$
	      with the convention that $F_i(x)=0$ if $i<0$.
	\end{corol}
	\begin{proof}
	      We have
	      $$S_n({\nu}(t+t^{-1})) = {\nu}^n(\sum_{k=0}^n t^{n-2k})$$
	      If $n$ is even, we have
	      \begin{align*}
		      {\nu}^n(\sum_{k=0}^n t^{n-2k})
			      &=\nu^n \left(1+\sum_{k=0}^{n/2-1} t^{n-2k} + t^{2k-n}\right)\\
			      &=\nu^n + \sum_{k=0}^{n/2-1} {\nu}^{2k}({\nu}^{n-2k}t^{n-2k}+{\nu}^{n-2k}t^{2k-n})\\
			      &= \nu^n + \sum_{k=0}^{n/2-1} {\nu}^{2k}({\nu}^{n-2k}t^{n-2k}+{\nu}^{n-2k}t^{2k-n})\\
			      &= \nu^n + \sum_{k=0}^{n/2-1} q^kF_{n-2k}({\nu}(t+t^{-1}))\\
			      &= \sum_{k=0}^n q^kF_{n-2k}({\nu}(t+t^{-1}))
	      \end{align*}
	      where the last equality follows from the convention that $F_i=0$ for $i<0$.
	      
	      If $n$ is odd, we denote by $[n/2]$ the integral floor of $n/2$. Then, 
	      \begin{align*}
		      {\nu}^n(\sum_{k=0}^n t^{n-2k})
			      &=\nu^n \sum_{k=0}^{[n/2]} t^{n-2k} + t^{2k-n}\\
			      &=\sum_{k=0}^{[n/2]} {\nu}^{2k}({\nu}^{n-2k}t^{n-2k}+{\nu}^{n-2k}t^{2k-n})\\
			      &=\sum_{k=0}^{[n/2]} {\nu}^{2k}({\nu}^{n-2k}t^{n-2k}+{\nu}^{n-2k}t^{2k-n})\\
			      &=\sum_{k=0}^{[n/2]} q^kF_{n-2k}({\nu}(t+t^{-1}))\\
			      &= \sum_{k=0}^n q^kF_{n-2k}({\nu}(t+t^{-1}))
	      \end{align*}
	      where the last equality follows from the convention that $F_i=0$ for $i<0$. This proves the corollary.
	\end{proof}

	\begin{corol}\label{corol:firstassecond}
	      For any $n \geq 0$, we have 
	      $$F^q_n(x)=S^q_n(x)-qS^q_{n-2}(x).$$
	\end{corol}
\end{section}

\begin{section}{Chebyshev polynomials and bases in affine cluster algebras}\label{section:bases}
	In this section, we study possible interactions between Chebyshev polynomials and bases in affine cluster algebras. We first study the case of the Kronecker quiver relying on \cite{shermanz, CZ}. Then, we study the case of a quiver of type $\Aaffine_{2,1}$ studied in particular in \cite{Cerulli:A21}. Finally we give conjectures for the general affine case.

	\begin{subsection}{The Kronecker quiver}
	      We consider the Kronecker quiver~:
	      $$\xymatrix{
		      K: & 1 \ar@<+2pt>[r]\ar@<-2pt>[r] & 2
	      }$$
	      We denote by $\mathcal A(K, \textbf x, \textbf y)$ the cluster algebra with principal coefficients at the initial seed $(K,\textbf x, \textbf y)$ where $\textbf y=\ens{y_1,y_2}$ and $\textbf x=\ens{x_1,x_2}$. We simply denote by $\mathcal A(K)$ the coefficient-free cluster algebra with initial seed $(K,\textbf x)$.
	      
	      $K$ is an affine quiver and the regular components of $\Gamma(kK)$ form a $\P^1(k)$-family of homogeneous tubes. The minimal imaginary root of $K$ is $\delta=\alpha_1+\alpha_2$. For any $\lambda \in \P^1(k)$, we denote by $\mathcal T_{\lambda}$ the tube corresponding to the parameter $\lambda$, by $M_\lambda$ the unique quasi-simple module in $\mathcal T_\lambda$ and for any $n \geq 1$, by $M_\lambda^{(n)}$ the indecomposable module in $\mathcal T_\lambda$ with quasi-socle $M_\lambda$ and quasi-length $n$. It follows from \cite{CZ} (see also \cite{Dupont:BaseAaffine}) that $X^Q_{M_\lambda}$ and $\XQy_{M_\lambda}$ do not depend on the choice of the parameter $\lambda \in \P^1(k)$ and we denote by
	      $$u(\textbf y)=\XQy_{M_\lambda}=\frac{x_1^2+y_1y_2x_2^2+y_2}{x_1x_2}$$
	      $$u=X^Q_{M_\lambda}=\frac{x_1^2+x_2^2+1}{x_1x_2}$$
	      these values.
	      
	      The following theorems give bases in the coefficient-free cluster algebra $\mathcal A(K)$.
	      \begin{theorem}[\cite{CZ}]\label{theorem:baseCZ}
		      The set
		      $$\mathcal C(K)=\ens{\textrm{cluster monomials}} \sqcup \ens{S_n(u) | n \geq 1}$$
		      is a $\Z$-basis in $\mathcal A(K)$.
	      \end{theorem}
	      \begin{theorem}[\cite{shermanz}]\label{theorem:baseSZ}
		      The set
		      $$\mathcal B(K)=\ens{\textrm{cluster monomials}} \sqcup \ens{F_n(u) | n \geq 1}$$
		      is a canonically positive $\Z$-basis in $\mathcal A(K)$. 
	      \end{theorem}
	      
	      In \cite[Theorem 6.3]{shermanz}, the authors defined a canonically positive basis in the cluster algebra with universal coefficients associated to the Kronecker quiver by lifting the canonically positive basis in the coefficient-free cluster algebra $\mathcal A(K,\textbf x)$. It is not clear at first whereas this lifting still gives rise to quantized Chebyshev polynomials. We prove that up to a normalization by coefficients, this is still the case for principal coefficients.
	      
	      For rank two cluster algebras, the description of coefficients used in \cite{shermanz} is the one introduced in \cite[Remark 2.5]{cluster1} which we will briefly recall here. Let $\mathcal A(K,\textbf x,\textbf y)$ be the cluster algebra with principal coefficients at the initial seed $(K,\textbf x,\textbf y)$ and $\P=\Trop(\textbf y)$. Then, $\P$ can be described as the free abelian group generated by $\ens{q_i | i \in \Z} \sqcup \ens{r_0,r_1}$ with respect to the relations
	      \begin{equation}\label{eq:relationqr}
		      r_{m-1}r_{m+1}=q_{m-1}q_{m+1}r_m^2
	      \end{equation}
	      where coefficients in terms of $\textbf y$ can be recovered from $\ens{r_i,q_i| i \in \Z}$ (see \cite[Remark 2.7]{cluster4}). In particular, we have the following equalities
	      \begin{equation}\label{eq:relationsqry}
		      r_1=y_2, r_2=1, q_1=1 \textrm{ and } q_2=y_1.
	      \end{equation}
	      
	      We now consider the completion $\hat \P=\Q \otimes_\Z \P$ obtained by adjoining to $\P$ the roots of all degrees from all the elements of $\P$. We denote by $\hat {\mathcal A}(K,\textbf x,\textbf y)$ the $\Z\hat \P$-algebra obtained from $\hat{\mathcal A}(K,\textbf x,\textbf y)$ by extension of scalars. An element in $\hat{\mathcal A}(K,\textbf x,\textbf y)$ is called \emph{positive} if it can be written as a Laurent polynomial with coefficients in $\Z_{>0}\hat \P$ in any cluster of $\hat{\mathcal A}(K,\textbf x,\textbf y)$. A $\Z\hat \P$-basis $\mathcal B$ of $\hat{\mathcal A}(K,\textbf x,\textbf y)$ is called \emph{canonically positive} if positive elements in $\hat{\mathcal A}(K,\textbf x,\textbf y)$ are exactly $\Z_{>0}\hat \P$-linear combinations of elements of $\mathcal B$. Such a basis is unique up to normalization by elements of $\hat \P$.
	      
	      \begin{theorem}
		      The set
		      $$\mathcal B^{\textbf y}(K)=\ens{\textrm{cluster monomials}} \sqcup \ens{F^{\textbf y^{\delta}}_n(u(\textbf y)) | n \geq 1}$$
		      is a canonically positive $\Z\hat \P$-basis in $\hat {\mathcal A}(K,\textbf x, \textbf y)$.
	      \end{theorem}
	      \begin{proof}
		      Let $\mathcal A(K,\u{\textbf x})$ be the coefficient-free cluster algebra with initial seed $(K,\u{\textbf x})$ where $\u {\textbf x}=\ens{\u x_1, \u x_2}$. 
		      We denote by $\u x_i, i \in \Z$ the cluster variables in the coefficient-free cluster algebra $\mathcal A(K,\u{\textbf x})$ given by $\u x_{i-1} \u x_{i+1}=\u x_i^2+1$ for any $i \in \Z$. 
		      
		      Set
		      $$u=\frac{\u x_1^2+\u x_2^2+1}{\u x_1 \u x_2}.$$
		      By Theorem \ref{theorem:baseSZ}, the canonically positive $\Z$-basis of $\mathcal A(K,\u{\textbf x})$ is given by
		      $$\mathcal B(K)=\ens{\textrm{cluster monomials}} \sqcup \ens{F_n(u) | n \geq 1}.$$
		      
		      Sherman and Zelevinsky proved in \cite{shermanz} that there exists a $\Z\hat \P$-linear isomorphism
		      $$\psi:\left\{\begin{array}{rcl}
			      \mathcal A(K,\u{\textbf x}) \otimes \Z\hat \P & \fl & \hat {\mathcal A}(K,\textbf x, \textbf y)\\
			      \u x_m & \mapsto & \left(\frac{q_m}{r_m}\right)^{\frac{1}{2}} x_m
		      \end{array}\right.$$
		      In particular, as $\psi$ is an isomorphism of $\Z\hat \P$-algebras, $\psi(\mathcal B(K))$ is a canonically positive $\Z\hat \P$-basis in $\hat {\mathcal A}(K,\textbf x, \textbf y)$. We now prove that, up to normalization by elements of $\hat \P$, $\psi(\mathcal B(K))$ and $\mathcal B^{\textbf y}(K)=\ens{\textrm{cluster monomials}} \sqcup \ens{F^{\textbf y^{\delta}}_n(u(\textbf y)) | n \geq 1}$ coincide. For this, it suffices to prove that for any $n \geq 1$, $\psi(F_n(u)) \in \hat \P F_n^{\textbf y^{\delta}}(u(\textbf y))$. More precisely, we prove by induction that for any $n \geq 1$, we have
		      $$\psi(F_n(u))=\textbf y^{-\frac{n\delta}{2}}F_n^{\textbf y^{\delta}}(u(\textbf y)).$$
		      We recall that 
		      $$u=\u x_0\u x_3-\u x_1\u x_2$$
		      where $\u x_0=\frac{\u x_1^2+1}{\u x_2}$ and $\u x_3=\frac{\u x_2^2+1}{\u x_1}$.
		      Expressing $r_0,r_3,q_0,q_3$ in terms of $y_1,y_2$ using identities (\ref{eq:relationqr}) and (\ref{eq:relationsqry}), a direct computation leads to 
		      \begin{align*}
			      \psi(u) 
				      &=\psi(\u x_0)\psi(\u x_3)-\psi(\u x_1)\psi(\u x_2)\\
				      &=\frac{1}{x_1 x_2}
				      \left(
					      \frac{x_1^2}{\textbf y^{\frac \delta 2}}+ \textbf y^{\frac \delta 2} x_2^2 + \frac{y_2^{-\frac 12}}{y_1^{-\frac 12}}
				      \right)\\
				      &=\textbf y^{-\frac \delta 2} \frac{x_1^2+y_1y_2x_2^2+y_2}{x_1x_2}\\
				      &=\textbf y^{-\frac \delta 2} u(\textbf y)
		      \end{align*}
		      Fix now some integer $n \geq 1$, we know that
		      $$F_{n+1}(u)=F_n(u)F_{1}(u)-F_{n-1}(u)$$
		      so 
		      \begin{align*}
			      \psi(F_{n+1}(u))
				      &=\psi(F_{n}(u))\psi(F_{1}(u))-\psi(F_{n-1}(u))\\
				      &=\textbf y^{-\frac{n\delta}{2}} F_{n}^{\textbf y^\delta}(u(\textbf y))\textbf y^{-\frac{\delta}{2}} F_{1}^{\textbf y^\delta}(u(\textbf y))-\textbf y^{-\frac{(n-1)}{2}\delta} F_{n-1}^{\textbf y^\delta}(u(\textbf y))\\
				      &=\textbf y^{-\frac{n+1}{2}\delta} F_{n}^{\textbf y^\delta}(u(\textbf y))F_{1}^{\textbf y^\delta}(u(\textbf y))-\textbf y^{-\frac{(n-1)}{2}\delta} F_{n-1}^{\textbf y^\delta}(u(\textbf y))\\
				      &=\textbf y^{-\frac{n+1}{2}\delta} \left(F_{n}^{\textbf y^\delta}(u(\textbf y))F_{1}^{\textbf y^\delta}(u(\textbf y))-\textbf y^{\delta} F_{n-1}^{\textbf y^\delta}(u(\textbf y))\right)\\
				      &=\textbf y^{-\frac{n+1}{2}\delta} \left(F_{n+1}^{\textbf y^\delta}(u(\textbf y))\right)
		      \end{align*}
		      where the last equality follows from Definition \ref{defi:firstkind}. Since $\hat \P$ contains the roots of all degrees of elements of $\P$, $\textbf y^{-\frac{n+1}{2}\delta}=y_1^{-\frac {n+1}{2}}y_2^{-\frac {n+1}{2}}$ belongs to $\hat \P$. For every $n \geq 1$, we have $F_n^{\textbf y^{\delta}}(u(\textbf y)) \in \P \psi(F_n(u))$ and thus up to a normalization by elements of $\hat \P$, $\psi(\mathcal B(K))$ coincides with $\mathcal B^{\textbf y}(K)$ and $\mathcal B^{\textbf y}(K)$ is a canonically positive basis of $\hat {\mathcal A}(K,\textbf x, \textbf y)$.
	      \end{proof}

	      Using Corollary \ref{corol:secondasfirst}, we get an analogue with coefficients of Theorem \ref{theorem:baseCZ}~:
	      \begin{theorem}
		      The set
		      $$\mathcal C^{\textbf y}(K)=\ens{\textrm{cluster monomials}} \sqcup \ens{S^{\textbf y^{\delta}}_n(u(\textbf y)) | n \geq 1}$$
		      is a $\Z\hat \P$-basis in $\hat{\mathcal A}(K,\textbf x, \textbf y)$.
	      \end{theorem}
	\end{subsection}

	\begin{subsection}{Quiver of type $\Aaffine_{2,1}$}
	      We now consider the quiver of type $\Aaffine_{2,1}$ from Example \ref{exmp:A21}~:
	      $$\xymatrix{
			      && 2\ar[rd]\\
		      Q: & 1 \ar[rr]\ar[ru] && 3.
	      }$$
	      We denote by $\mathcal A(Q, \textbf x, \textbf y)$ the cluster algebra with principal coefficients at the initial seed $(Q,\textbf x, \textbf y)$ where $\textbf y=\ens{y_1,y_2,y_3}$ and $\textbf x=\ens{x_1,x_2,x_3}$. We denote by $\mathcal A(Q,\textbf x)$ the coefficient-free cluster algebra with initial seed $(Q,\textbf x)$. Keeping notations of Example \ref{exmp:A21}, in the coefficient-free settings we have~:
	      \begin{theorem}[\cite{Dupont:BaseAaffine}]\label{theorem:basesemicanA21}
		      The set
		      $$\mathcal C(Q)=\ens{\textrm{cluster monomials}} \sqcup \ens{S_n(u)z^k,S_n(u)w^k | n \geq 1,k \geq 0}$$
		      is a $\Z$-basis in $\mathcal A(Q,\textbf x)$.
	      \end{theorem}
	      
	      \begin{rmq}
		      Actually, the basis defined in \cite{Dupont:BaseAaffine} is given by 
		      $$\mathcal S(Q)=\ens{\textrm{cluster monomials}} \sqcup \ens{u^nz^k,u^nw^k | n \geq 1,k \geq 0}.$$
		      Since for every $n \geq 1$, $S_n(u)$ is a monic polynomial in $u$ of degree $n$, Theorem \ref{theorem:basesemicanA21} follows directly from the above statement. 
	      \end{rmq}

	      With coefficients, Cerulli proved the following theorem~:
	      \begin{theorem}[\cite{Cerulli:A21}]
		      The set
		      $$\mathcal B^{\textbf y}(Q)=\ens{\textrm{cluster monomials}} \sqcup \ens{F^{\textbf y^{\delta}}_n(u(\textbf y))w(\textbf y)^k, F^{\textbf y^{\delta}}_n(u(\textbf y))z(\textbf y)^k | n \geq 1, k \geq 0}$$
		      is a canonically positive $\Z\P$-basis in $\mathcal A(Q, \textbf x, \textbf y)$.
	      \end{theorem}
	      
	      It then follows directly from Corollary \ref{corol:secondasfirst} that~:
	      \begin{corol}
		      The set
		      $$\mathcal C^{\textbf y}(Q)=\ens{\textrm{cluster monomials}} \sqcup \ens{S^{\textbf y^\delta}_n(u(\textbf y))z(\textbf y)^k,S^{\textbf y^\delta}_n(u(\textbf y))w(\textbf y)^k | n \geq 1,k \geq 0}$$
		      is a $\Z\P$-basis in $\mathcal A(Q,\textbf x,\textbf y)$.
	      \end{corol}
	\end{subsection}
	      
	\begin{subsection}{Conjectures for the general affine case}
	      Let $Q$ be an affine quiver. Let $\lambda$ be the parameter of an homogeneous tube in $\Gamma(kQ)$, we know that $\XQy_{M_\lambda}$ does not depend on the chosen parameter $\lambda$ (see e.g. \cite{DXX:basesv3}) and we denote by $z=\XQy_{M_\lambda}$ this common value. We denote by $\mathcal E_{\mathcal R}$ the set of rigid regular modules in $kQ$-mod. Generalizing results of \cite{Dupont:BaseAaffine} for the coefficient-free case, we conjecture that a $\Z\P$-basis in the cluster algebra $\mathcal A(Q,\textbf x, \textbf y)$ with principal coefficients at the initial seed $(Q,\textbf x, \textbf y)$ can be described as follows~:
	      \begin{conj}\label{conj:semican}
		      Let $Q$ be an affine quiver, then 
		      $$\mathcal C^{\textbf y}(Q)=\ens{\textrm{cluster monomials}} \sqcup \ens{S_n^{\textbf y^\delta}(z)\XQy_R | n \geq 1, R \in \mathcal E_{\mathcal R}}$$
		      is a $\Z\P$-basis in $\mathcal A(Q,\textbf x, \textbf y)$.
	      \end{conj}
	      
	      According to the previous examples, we also give a conjecture for a canonically positive basis in an affine cluster algebra~:
	      \begin{conj}\label{conj:can}
		      Let $Q$ be an affine quiver, then 
		      $$\mathcal B^{\textbf y}(Q)=\ens{\textrm{cluster monomials}} \sqcup \ens{F_n^{\textbf y^\delta}(z)\XQy_R | n \geq 1, R \in \mathcal E_{\mathcal R}}$$
		      is a canonically positive $\Z\P$-basis in $\mathcal A(Q,\textbf x, \textbf y)$.
	      \end{conj}
	      
	      In the above examples, the general idea is thus that quantized Chebyshev polynomials of the second kind are related to representation theoretic properties whereas quantized Chebyshev polynomials of the first kind are related to positivity properties. More precisely, if $\mathcal A(Q,\textbf y, \textbf x)$ is a cluster algebra associated to an affine quiver $Q$, quantized Chebyshev polynomials of the second kind arise naturally from the study of $\mathcal A(Q,\textbf y, \textbf x)$ through the representation theory of $Q$. Using methods proposed in \cite{CK1,Dupont:BaseAaffine, mathese, DXX:basesv3}, it thus seems likely to prove that $\mathcal C^{\textbf y}(Q)$ is a $\Z\P$-basis in $\mathcal A(Q,\textbf x, \textbf y)$. Then, using Corollary \ref{corol:secondasfirst}, one can express quantized Chebyshev polynomials of the second kind as positive $\Z\P$-linear combinations of quantized Chebyshev polynomials of the first kind. Thus, if Conjecture \ref{conj:semican} is proved, it is straightforward to prove that $\mathcal B^{\textbf y}(Q)$ is a $\Z\P$-basis. The remaining part would thus be to prove that $\mathcal B^{\textbf y}(Q)$ is a canonically positive $\Z\P$-basis. 
	\end{subsection}
\end{section}

\section*{Acknowledgements}
The author would like to thank Giovanni Cerulli Irelli for stimulating discussions about affine cluster algebras of type $\Aaffine_{2,1}$. The results presented in \cite{Cerulli:A21} provided the motivation for introducing quantized Chebyshev polynomials. He would also like to thank Philippe Caldero for interesting remarks and ideas about this subject. Finally, he would like to thank Shih-Wei Yang for pointing out the connections of this work with \cite{YZ:ppalminors}.


\newcommand{\etalchar}[1]{$^{#1}$}

\end{document}